\documentclass[11pt]{amsart}
\usepackage{amsmath,amssymb,amsthm,enumerate}

\usepackage{xy} %for diagrams
\usepackage{mathrsfs} %for script fonts
\usepackage{graphicx}
\usepackage[pagebackref=true]{hyperref}

\title{On elementary amenable bounded automata groups}
\author{Kate Juschenko}
\author{Benjamin Steinberg}\thanks{The second author was supported by  NSA MSP \#H98230-16-1-0047.}
\author{Phillip Wesolek}
\date{\today}

\usepackage[usenames, dvipsnames]{color}

\newtheorem{thm}{Theorem}[section]

\newtheorem{obs}[thm]{Observation}
\newtheorem{prop}[thm]{Proposition}
\newtheorem{lem}[thm]{Lemma}
\newtheorem{cor}[thm]{Corollary}
\newtheorem{conj}[thm]{Conjecture}
\theoremstyle{definition}
\newtheorem{defn}[thm]{Definition}
\newtheorem{quest}[thm]{Question}
\newtheorem{rmk}[thm]{Remark}

\newtheorem*{claim*}{Claim}

\newtheorem{example}[thm]{Example}

\newcommand{\Zb}{\mathbb{Z}}
\newcommand{\Nb}{\mathbb{N}}

\newcommand{\mc}[1]{\mathcal{#1}}

\newcommand{\acts}{\curvearrowright}

\newcommand{\Aut}{\mathrm{Aut}}
\newcommand{\sym}{\mathrm{Sym}}
\newcommand{\rist}{\mathrm{rist}}

\newcommand{\normal}{\trianglelefteq}
\newcommand{\rest}{\upharpoonright}
\newcommand{\rk}{\mathrm{rk}}

\newcommand{\Xt}{X^{*}}
\newcommand{\Sym}{\mathrm{Sym}}

\newcommand{\injects}{\hookrightarrow}
 
 \newcommand{\EA}{\mathrm{EG}}
 
 \newcommand{\Supp}{\mathrm{Supp}}
 \newcommand{\fix}{\mathrm{fix}}
 \newcommand{\at}{\widetilde{a}}
 \newcommand{\bt}{\widetilde{b}}
\newcommand{\grp}[1]{\langle #1 \rangle}

\begin{document}

\begin{abstract}
There are several natural families of groups acting on rooted trees for which every member is known to be amenable. It is, however, unclear what the elementary amenable members of these families look like. Towards clarifying this situation, we here study elementary amenable bounded automata groups. We are able to isolate the elementary amenable bounded automata groups in three natural subclasses of bounded automata groups. In particular, we show that iterated monodromy groups of post-critically finite polynomials are either virtually abelian or not elementary amenable.
\end{abstract}

\maketitle
\tableofcontents
\addtocontents{toc}{\protect\setcounter{tocdepth}{1}}

\section{Introduction}

A group $G$ is amenable if there exists a finitely additive translation invariant probability measure on all subsets of $G$. This definition was given by J. von Neumann, \cite{von-neumann1}, in response to the Banach-Tarski and Hausdorff paradoxes. He singled out the property of a group which forbids paradoxical actions.

The class of {\it elementary amenable groups}, denoted by  $EG$, was introduced by M. Day in \cite{day:semigroups} and is defined to be the smallest class of groups that contains the finite groups and the abelian groups and is closed under taking subgroups, quotients, extensions, and directed unions. The fact that the class of amenable groups is closed under these operations was already known to von Neumann, \cite{von-neumann1}, who noted at that at that time there was no known amenable group which did not belong to $EG$.

No substantial progress in understanding the class of elementary amenable groups was made until the~80s. C. Chou, \cite{C}, showed that all elementary amenable groups have either polynomial or exponential growth. Shortly thereafter, R. Grigorchuk, \cite{grigorchuk:milnor_en}, gave an example of a group with intermediate growth, and such groups are necessarily amenable but by Chou's work are not elementary amenable. Grigorchuk's group served as a starting point in developing the theory of groups with intermediate growth.

In the same paper, Chou showed that every infinite finitely generated simple group is not elementary amenable. In \cite{JM}, it was shown that the  topological full group of a Cantor minimal system is amenable, and by the results of H. Matui, \cite{Matui}, this group has a simple and finitely generated commutator subgroup, in particular, it is not elementary amenable. This was the first example of an infinite finitely generated simple amenable group. The strategy of Matui has been extended in \cite{cjn}  and \cite{nek1} to minimal actions of $\mathbb{Z}^d$ and groupoids, respectively. This has produced more examples of finitely generated simple amenable groups in combination with results of \cite{JMMS}, where it is demonstrated that there are minimal actions of $\mathbb{Z}^2$ on the Cantor set with amenable topological full groups.

Groups of automorphisms of rooted trees were the only source of groups of intermediate growth until a recent result of V. Nekrashevych, \cite{nek2}, who found finitely generated, torsion  subgroups of the topological full group of intermediate growth. These groups cannot act on rooted tree, since they are simple.

Summarizing the above, the currently known sources of non-elementary amenable groups are  groups acting on rooted trees and topological full groups of Cantor minimal systems.

There has been considerable progress in establishing amenability of groups acting on rooted trees. L. Bartholdi, V. Kaimanovich, and Nekrashevych in \cite{BKN10} proved that groups acting on rooted trees with {\it  bounded activity} are amenable; these are the so-called \textit{bounded automata groups}. This was extended to groups of linear activity by G. Amir, O. Angel, and B. Vir\'ag in \cite{AAV}, and in \cite{JNS}, Nekrashevych, M. de la Salle, and the first named author extended this result further to groups of quadratic activity. These results produce large families of amenable groups, but they do not identify the non-elementary amenable members of the families.  The first named author established what seems to be the first results studying the elementary amenable groups that act on rooted trees in \cite{Ju15}. The present article picks up this thread of research again. We here consider the elementary amenable bounded automata groups.

\begin{quest}\label{qu:main}
What do elementary amenable bounded automata groups look like? How close are they to abelian?
\end{quest}

Question~\ref{qu:main} in full generality seems out of reach at the moment. We answer here the question for three large classes of bounded automata groups, which contain many of the hitherto studied examples. Two of these classes are new, and they appear to be interesting in their own right. We additionally present several examples which place restrictions on how one might specify ``close to abelian."

\subsection{Iterated monodromy groups}
The iterated monodromy groups of post-critically finite polynomials, see \cite[Chapter 5]{N05}, form a compelling and natural class of bounded automata groups. The study of these groups has generated a significant body of research in both dynamics and group theory. See, for example, \cite{G12,nek3}.

For this class of groups, we answer the motivating question completely.
\begin{thm}[See Corollary~\ref{cor:monodromy}]
Every iterated monodromy group of a post-critically finite polynomial is either virtually abelian or not elementary amenable.
\end{thm}

\subsection{Generalized basilica groups}
For a finite set $X$, let $X^*$ be the free monoid on $X$. By identification with its Cayley graph, $X^*$ is naturally a rooted tree. For $g\in \Aut(\Xt)$ and $v\in \Xt$, the section of $g$ at $v$ is denoted by $g_v$; see Section~\ref{sec:prelim} for the relevant definitions.

\begin{defn}
We say that $G\leq\Aut(\Xt)$ is a \textbf{generalized basilica group} if $G$ admits a finite generating set $Y$ such that for every $g\in Y$ either $g_x=1$ for all $x\in X$ or $g_x\in \{1,g\}$  for all $x\in X$ with exactly one $x$ such that $g_x=g$.
\end{defn}

\begin{rmk}
The classical basilica group can be represented as a generalized basilica group. The representation is given by considering the natural action of the basilica group on $[4]^*$, the $4$-regular rooted tree, by restricting the action on the binary tree to even levels; see Example~\ref{ex:basilica}.
\end{rmk}

A key feature of generalized basilica groups is a reduction theorem, Theorem~\ref{thm:reduction_thm_GB},  which reduces the analysis to self-replicating groups.  Using this reduction, we answer the motivating question for a class of generalized basilica groups for which we have some control over the ``cycle structure" of the generators.
\begin{defn}
Let $G=\grp{Y}\leq \Aut(\Xt)$ be a generalized basilica group. We say that $G$ is \textbf{balanced} if for each $g\in Y$ such that there is $x\in X$ for which $g_x=g$, the least $n\geq 1$ for which $g^n$ fixes $x$ is also such that $g^n$ fixes $X$.
\end{defn}
Note that the basilica group acting on $[4]^*$ is balanced.

\begin{thm}[See Corollary~\ref{cor:balanced}]
Every balanced generalized basilica group is either (locally finite)-by-(virtually abelian) or not elementary amenable.
\end{thm}
Examples~\ref{ssec:balanced} show our theorem for balanced groups cannot be sharpened. While a proof eludes us, we expect the following conjecture to hold.
\begin{conj}
Every generalized basilica group is either (locally finite)-by-(virtually abelian) or not elementary amenable
\end{conj}

The class of generalized basilica groups encompasses the torsion-free bounded automata groups, suggesting it is a natural class.

\begin{prop}[See Corollary~\ref{cor:torsion-free_GBG}]
If $G\leq \Aut(\Xt)$ is a torsion-free bounded automata group, then there is a torsion-free generalized basilica group $H$ and $k\geq 1$ such that $H$ is isomorphic to a subgroup of $G$ and
\[
G\injects  \Sym(X^k)\ltimes H^{X^k}.
\]
\end{prop}
\begin{prop}[See Corollary~\ref{cor:torsion_free}]
Every torsion-free self-replicating bounded automata group admits a faithful representation as a generalized basilica group.
\end{prop}
We also exhibit countably many non-isomorphic balanced generalized basilica groups; see Examples~\ref{ssec:balanced}. The class of balanced generalized basilica groups, and hence the class of generalized basilica groups, is thus as large as it can be, since there are only countably many bounded automata groups.

\subsection{Groups of abelian wreath type}
Let $\pi_1:\Aut(\Xt)\rightarrow \Sym(X)$ be the homomorphism induced by the action of $\Aut(\Xt)$ on $X$. Under the natural identification with finitary elements, we regard $\Sym(X)\leq \Aut(\Xt)$, when convenient.

\begin{defn}
 We say that $G\leq \Aut(\Xt)$ is of \textbf{abelian wreath type} if $\pi_1(G)\leq \Sym(X)$ is abelian and $G$ admits a finite self-similar generating set $Y$ such that for every $g\in Y$ either $g_x=1$ for all $x\in X$ or $g_x\in \{g\}\cup \pi_1(G)$ for all $x\in X$ with exactly one $x$ such that $g_x=g$.
\end{defn}
\begin{rmk}
The Gupta-Sidki groups and the more general GGS groups are of abelian wreath type.
\end{rmk}

A group $G\leq \Aut(\Xt)$ of abelian wreath type embeds into the inverse limit $(\dots \wr_{X} A\wr_{X} A)$ where $A$ is the abelian group $\pi_1(G)$. This class of bounded automata groups seems to be among the simplest such groups, because the structure of these groups is minimally complicated by permutation-group-theoretic phenomena.

We answer the motivating question completely for self-replicating groups of abelian wreath type.

\begin{thm}[See Theorem~\ref{thm:aws_groups}]\label{thm:aws_groups_intro}
Every self-replicating group of abelian wreath type is either virtually abelian or not elementary amenable.
\end{thm}

The infinite dihedral group shows that Theorem~\ref{thm:aws_groups} is sharp; see Example \ref{ssec:dihedral}.  The GSS groups are self-replicating as soon as they are infinite, so Theorem~\ref{thm:aws_groups_intro} applies to these groups. Additionally, as there are countably many self-replicating GGS groups, there are countably many isomorphism types of self-replicating groups of abelian wreath type.

\subsection{Groups containing odometers}

An element $d\in \Aut(\Xt)$ is said to be an \textbf{odometer} if $\grp{d}$ acts transitively on $X$ and $d_x\in \{1,d\}$ for all $x\in X$ with exactly one $x\in X$ such that $d_x=d$. Odometers arise somewhat naturally in bounded automata groups. For instance some iterated monodromy groups will contain these; see Example~\ref{ex:basilica_2}. On the other hand, our theorem for groups with odometers plays an important technical role in other arguments, so one is welcome to regard this theorem as a useful technical tool.

\begin{thm}[See Theorem~\ref{thm:odometers}]\label{thm:trans_intro}
Every bounded automata group that contains an odometer is either virtually abelian or not elementary amenable.
\end{thm}

Let $G$ be a bounded automata group containing a level transitive element $d$.  The element $d$ is well-known to be conjugate in the full automorphism group of the rooted tree to an odometer, but the conjugator need not preserve the property of $G$ being a bounded automata group.   In Theorem~\ref{thm:trans_intro}, we need the element to be an odometer in the bounded representation of $G$.

\section{Preliminaries}\label{sec:prelim}
For a group $G$ acting on a set $X$ and $Y\subseteq X$, we write $G_{(Y)}$ for the pointwise stabilizer of $Y$ in $G$. For $g,h\in G$, the commutator $[g,h]$ is $ghg^{-1}h^{-1}$.

\subsection{Bounded automata groups}\label{ssec:BAG}
For a finite set $X$ of size at least $2$, let $X^*$ be the free monoid on $X$; the identity is the empty word. By identification with its Cayley graph, $X^*$ is naturally a rooted tree, but in practice and in the work at hand, it is often more useful to think of $\Xt$ as words. We write $X^k$ to denote the words of length $k$. An automorphism $g$ of $X^*$ is a bijection such preserves the prefix relation. That is, $a$ is a prefix of $b$ if and only if $g(a)$ is a prefix of $g(b)$. Automorphisms of $X^*$ act on $X^*$ from the left, and we denote the group of automorphisms by $\Aut(\Xt)$.

For each $k\geq 1$, the group $\Aut(\Xt)$ acts on the set of words of length $k$. We let $\pi_k:\Aut(\Xt)\rightarrow \Sym(X^k)$ denote the induced homomorphism. It is sometimes convenient to see $\pi_k(\Aut(\Xt))\leq \Aut(\Xt)$, and this is achieved as follows: Say $\sigma\in \pi_k(\Aut(\Xt))$ and $v\in X^l$ with $l\geq k$. Writing $v=xw$ with $x\in X^k$ and $w\in \Xt$, we define $\sigma(v):=\sigma(x)w$. One verifies that this identification gives that $\pi_k(\Aut(\Xt))\leq \Aut(\Xt)$.

For a word $w\in \Xt$, the words with prefix $w$, $w\Xt$, are isomorphic to $\Xt$ by the map $\psi_w:\Xt\rightarrow w\Xt$ defined by $\psi_w:x\mapsto wx$. For $g\in \Aut(\Xt)$ and $w\in \Xt$, the element $g$ induces an automorphism $g_w\in \Aut(\Xt)$ defined by
\[
g_w:=\psi_{g(w)}^{-1}\circ g\circ \psi_w.
\]
The automorphism $g_w$ is called the \textbf{section} of $g$ at $w$. A straightforward computation shows the operation of taking sections enjoys two fundamental properties.

\begin{lem} Let $X$ be a non-empty finite set, $g$, $h$, and $k$ be elements of $\Aut(\Xt)$, and $v_1$ and $v_2$ be elements of $\Xt$. Then,
	\begin{enumerate}[(a)]
		\item $g_{v_1v_2}=(g_{v_1})_{v_2}$
		\item $(hk)_v={h}_{k(v)}{k}_v$
	\end{enumerate}
\end{lem}
\begin{rmk} The previous lemma ensures that map $\Aut(\Xt)\times \Xt\rightarrow \Aut(\Xt)$ by $(g,v)\mapsto g_v$ is a cocycle.
\end{rmk}

An automorphism $g\in \Aut(\Xt)$ is called \textbf{finitary} if there is some $k\geq 1$ such that $g_v=1$ for all $v\in X^k$. The least $k$ such that $g_v=1$ for all $v\in X^k$ is called the \textbf{depth} of the finitary element $g$. An automorphism $g\in \Aut(\Xt)$ is called \textbf{directed} if there is some $k\geq 1$ and $w\in X^k$ such that $g_w=g$ and $g_v$ is finitary for all $v\in X^k\setminus \{w\}$. The least such $k$ is called the \textbf{period} of $g$, and the vertex $w\in X^k$ is called the \textbf{active vertex} of $g$ on $X^k$ and is denoted by $v^g$. We say that $g$ is \textbf{strongly active} if $g(v^g)\neq v^g$. We say $g$ is \textbf{strongly directed} if there exists $k\geq 1$ and $w\in X^k$ such that $g_w=g$ and $g_u=1$ for all $u\in X^k\setminus \{w\}$. An element $d\in \Aut(\Xt)$ is said to be an \textbf{odometer} if $\grp{d}$ acts transitively on $X$ and $d_x\in \{1,d\}$ for all $x\in X$ with exactly one $x\in X$ such that $d_x=d$.

For the work at hand, the following features of automorphisms of rooted trees play an integral role.
\begin{defn}
An automorphism $g\in \Aut(\Xt)$ is \textbf{finite state} if $S(g):=\{g_v\mid v\in \Xt\}$ is finite.  The element $g$ is \textbf{bounded} if there is $N\geq 0$ such that for all $n\geq 1$,
\[
|\{v\in X^v\mid g_v\neq 1\}|\leq N.
\]
\end{defn}

Finitary automorphisms and directed automorphisms are clearly bounded and finite state. A partial converse also holds.
\begin{prop}[{\cite[Proposition 2.7]{BKN10}}]\label{prop:directed_depth}
An automorphism $g\in \Aut(\Xt)$ is bounded and finite state if and only if there exists $m\geq 1$ such that each section $g_v$ for $v\in X^m$ is either finitary or directed.
\end{prop}

We are now prepared to define the bounded automata groups.
\begin{defn}
	A (set) group $G\leq \Aut(\Xt)$ is \textbf{self-similar} if $g_w\in G$ for all $w\in \Xt$ and $g\in G$.
\end{defn}

\begin{defn}
A group $G\leq \Aut(\Xt)$ is said to be a \textbf{bounded automata group} if $G$ is finitely generated and self-similar and every element $g\in G$ is bounded and finite state.
\end{defn}
\begin{rmk}
Bounded automata groups are named as such, because they are exactly the groups such that the generators are given by a bounded automaton; see \cite[Chapter 1.5]{N05}.
\end{rmk}

\subsection{Self-replicating groups}
Letting $\Aut(\Xt)_{(w)}$ be the stabilizer of $w$ in $\Aut(\Xt)$, the map
\[
\phi_w:\Aut(\Xt)_{(w)}\rightarrow \Aut(\Xt)
\]
by $g\mapsto g_w$ is a homomorphism. We call $\phi_w$ the \textbf{section homomorphism} at $w$. For a group $G\leq \Aut(\Xt)$, we denote the image $\phi_w(G_{(w)})$ by $\sec_G(w)$, and $\sec_G(w)$ is called the \textbf{group of sections} of $G$ at $w$.

\begin{defn}
A group $G\leq \Aut(\Xt)$ is called \textbf{self-replicating} if $G$ is self-similar, acts transitively on $X$ and $\sec_G(x)=G$ for all $x\in X$.
\end{defn}
\begin{rmk}
The terms ``fractal" and ``recurrent" have previously appeared in the literature in place of ``self-replicating." We understand that ``self-replicating" is the currently accepted term.
\end{rmk}
For a self-replicating group $G\leq \Aut(\Xt)$, it follows by induction on the length of $v\in \Xt$ that $\sec_G(v)=G$ for all $v\in \Xt$. Self-replicating groups additionally act transitively on $X^k$ for every $k$, not just on $X$. Groups that act transitively on $X^k$ for all $k\geq 1$ are called \textbf{level transitive}. For $g\in \Aut(\Xt)$ an odometer, $\grp{g}$ is level transitive.

In the present work, weak forms of self-similarity and self-replication arise frequently.
\begin{defn}
We say $G\leq \Aut(X^*)$ is \textbf{weakly self-similar} if $\sec_G(v)\leq G$ for every $v\in \Xt$. We say that $G$ \textbf{weakly self-replicating} if $\sec_G(v)=G$ for every $v\in X^*$.
\end{defn}
A weakly self-similar (weakly self-replicating) group need not be self-similar (self-replicating). This holds even in the case that $G$ acts transitively on $X$. See Example~\ref{ssec:weak_self-rep}

\begin{rmk} In the work at hand, we will need to analyze subgroups of a bounded automata group which are weakly self-similar, but not self-similar. The expert will note that one can conjugate a weakly self-similar group to obtain a self-similar group. However, if one conjugates the supergroup to make the subgroup in question self-similar, the conjugate often fails to be a bounded automata group.
\end{rmk}

\subsection{Elementary amenable groups}

\begin{defn} The class of \textbf{elementary amenable groups}, denoted by $\EA$, is the smallest class of groups such that the following hold.
\begin{enumerate}
\item $\EA$ contains all abelian groups and  finite groups.
\item $\EA$ is closed under taking subgroups, group extensions, quotients, and direct limits.
\end{enumerate}
\end{defn}

By work of D. Osin, \cite{Os02}, the class $\EA$ admits a simpler description. Let us define recursively classes $\EA_{\alpha}$ for $\alpha$ an ordinal as follows.
\begin{enumerate}[$\bullet$]
\item $\EA_0$ is the class of abelian groups and finite groups.
\item Given $\EA_{\alpha}$, let $\EA^e_{\alpha}$ be the class of $\EA_{\alpha}$-by-$\EA_0$ groups and let $\EA^l_{\alpha}$ be the groups given by a direct limit of groups in $\EA_{\alpha}$. We define $\EA_{\alpha+1}:=\EA_{\alpha}^e\cup \EA_{\alpha}^l$.
\item For $\lambda$ a limit ordinal, we define $\EA_{\lambda}:=\bigcup_{\beta<\lambda} \EA_{\beta}$.
\end{enumerate}

\begin{thm}[{Osin, \cite[Theorem 2.1]{Os02}}] $\EA=\bigcup_{\alpha\in ORD} \EA_{\alpha}$.
\end{thm}

In view of Osin's theorem, the class of elementary amenable groups admits a well-behaved ordinal-valued rank.

\begin{defn} For $G\in \EA$, we define
\[
\rk(G):=\min\{\alpha\mid G\in \EA_{\alpha}\}
\]
and call $\rk(G)$ the \textbf{construction rank} of $G$.
\end{defn}

\begin{cor}\label{cor:EA-rank}
If $G\in \EA$ is finitely generated and neither finite nor abelian, then there is $L\normal G$ such that $G/L$ is finite or abelian and $\rk(G)=\rk(L)+1$.
\end{cor}

\begin{prop}\label{prop:rank_stability}
Let $G\in \EA$.
\begin{enumerate}
\item If $H\leq G$, then $\rk(H)\leq \rk(G)$. \textnormal{(\cite[Lemma 3.1]{Os02})}
\item If $L\normal G$, then $\rk(G/L)\leq \rk(G)$. \textnormal{(\cite[Lemma 3.1]{Os02})}
\item If $\{1\}\rightarrow K\rightarrow G\rightarrow Q\rightarrow \{1\}$ is a short exact sequence, then $\rk(G)\leq \rk(K)+\rk(Q)$.\textnormal{(\cite[cf. Lemma 3.2]{Os02})}
\item  $\rk(G\times G)=\rk(G)$. \textnormal{(folklore)}
\end{enumerate}
\end{prop}

\section{Reduced form}
A given representation of a bounded automata group may not always be optimal. The reduction developed herein produces a canonical representation of a bounded automata group which often makes arguments simpler. This technique appears to be folklore.

Given a tree $\Xt$ and $k\geq 1$, the group $\Aut(\Xt)$ acts faithfully on $(X^k)^*$, giving an embedding $r_k:\Aut(\Xt)\injects \Aut((X^k)^*)$. The map $r_k$ is just the restriction of the action of $\Aut(\Xt)$ to $(X^k)^*\subset \Xt$ (which is a free monoid on $X^k$ and hence an $|X|^k$-regular tree).

\begin{lem}\label{lem:directed_preserve} Suppose that $g\in \Aut(\Xt)$ and $k\geq 1$.
\begin{enumerate}
\item For any $w\in (X^k)^*$ and $g\in G$, $r_k(g_w)=(r_k(g))_w$.
\item If $g$ is finitary, then $r_k(g)$ is finitary.
\item If $g$ is directed, then $r_k(g)$ is directed.
\item If $g$ is strongly directed, then $r_k(g)$ is strongly directed.
\end{enumerate}
\end{lem}
\begin{proof}
	Set $O:=X^k$ and take $w\in O^*$. By definition, the section of $g$ at $w$ in $\Xt$ is $g_w=\psi_{g(w)}^{-1}\circ g\circ \psi_w$ where $\psi_v:X^*\rightarrow vX^*$ by $x\mapsto vx$. The map $r_k$ is restriction to $O^*$, so we see that
	\[
	r_k(g_w)=(g_w)\rest_{O^*}=\psi_{g(w)}^{-1}\rest_{O^*}\circ g\rest_{O^*}\circ \psi_w\rest_{O^*}=(r_k(g))_w.
	\]
Hence, (1) holds.

If $g$ is finitary, then we can find some $n\geq 1$ such that $g_w=1$ for all $w\in O^n$. Claim (1) shows that $(r_k(g))_w=1$ for all $w\in O^n$, so $r_k(g)$ is finitary, verifying (2).

If $g$ is (strongly) directed, we can find some $n\geq 1$ such that for all $w\in O^n$ either $g_w=g$ or $g_w$ is finitary (trivial). If $g_w$ is finitary (trivial), then claims (1) and (2) ensure that $(r_k(g))_w$ is finitary (trivial). If $g_w=g$, then (1) implies that $r_k(g)=r_k(g_w)=(r_k(g))_w$. We conclude that for every $w\in O^n$ either $(r_k(g))_w=r_k(g)$ or $r_k(g)_w$ is finitary (trivial). Hence, $r_k(g)$ is a (strongly) directed element of $\Aut(O^*)$, establishing (3) and (4).
\end{proof}

\begin{cor}\label{cor:preserve_BAG}
If $G\leq\Aut(\Xt)$ is a bounded automata group and $k\geq 1$, then $r_k(G)\leq\Aut((X^k)^*)$ is a bounded automata group.
\end{cor}
\begin{proof}
This is immediate from Proposition~\ref{prop:directed_depth} and Lemma~\ref{lem:directed_preserve}.
\end{proof}

Recall that for each $k\geq 1$, there is a homomorphism from $\Aut(\Xt)$ into $\Sym(X^k)$ denoted by $\pi_k:\Aut(\Xt)\rightarrow \Sym(X^k)$.  The functions  $\pi_k$ and $r_k$ enjoy the following relationship:
\[
\pi_1\circ r_k=\pi_k.
\]
Recall also that we may see $\pi_k(\Aut(\Xt))\leq \Aut(\Xt)$ whenever necessary.

\begin{defn}
For $G\leq\Aut(\Xt)$ a bounded automata group, we say that $G$ is in \textbf{reduced form} if $G$ admits a finite generating $Y$ such that for every $g\in Y$ either $g_x=1$ for all $x\in X$ or $g_x\in \{g\}\cup \pi_1(G)$ for all $x\in X$ with exactly one $x$ such that $g_x=g$. We say that $Y$ is a \textbf{distinguished generating set} for $G$. We write $G=\grp{Y}\leq\Aut(\Xt)$ to indicate that $Y$ is a distinguished generating set for $G$.
\end{defn}
Note that one may always assume that $Y$ is a self-similar set by adding each section from $\pi_1(G)$ of a directed element of $Y$ to $Y$.  Neither the reduced form nor the distinguished generating set for a given reduced form are unique in general. However, all bounded automata groups can essentially be put in reduced form. Establishing this requires a result from the literature.

\begin{thm}[{\cite[Theorem 3.9.12]{N05}}]\label{thm:gen_by_nuc}
For $G\leq \Aut(\Xt)$ a bounded automata group, there is a finite self-similar set $Z\subseteq G$ consisting of finitary and directed elements such that for all $g\in G$ there is $N$ for which $g_v\in Z$ for all $v\in X^k$ and $k\geq N$.
\end{thm}

The directed elements of the set $Z$ in Theorem~\ref{thm:gen_by_nuc} are all strongly directed if and only if in the automaton with state set $Z$ (and the obvious transitions), each non-trivial state belonging to a cycle cannot reach any state off that cycle except for the trivial state.

\begin{thm}\label{thm:reduction-reduced-form}
For $G\leq \Aut(\Xt)$ a bounded automata group, there are a finite set $Z\subseteq G$ and $k\geq 1$ such that the following hold:
\begin{enumerate}
\item $r_k(G)\leq\Aut((X^k)^*)$ is a bounded automata group.
\item Setting $H:=r_k(\grp{Z})$, $H\leq \Aut((X^k)^*)$ is a bounded automata group in reduced form with distinguished generating set $r_k(Z)$.  Moreover, if the directed elements of the set from Theorem~\ref{thm:gen_by_nuc} are strongly directed, then $H$ is generalized basilica.
\item There is an injective homomorphism $G\rightarrow \Sym(X^k)\ltimes H^{X^k}$.
\end{enumerate}
\end{thm}
\begin{proof}
Let $W$ be a finite generating set for $G$ and let $Z\subset G$  be as given by Theorem~\ref{thm:gen_by_nuc}; in particular, $Z$ is closed under taking sections. By taking a common multiple of the periods of directed $d\in Z$, we may find $m$ such that for each directed $d\in Z$ and $x\in X^m$ the section $d_x$ is either finitary or equal to $d$. Taking $n$ large enough, we may assume that $f_x=1$ for all $x\in X^n$ and finitary $f\in Z$.

Let $k$ be a sufficiently large multiple of $m$ so that $k\geq n$ and $g_v\in Z$ for all $g\in W$ and $v\in X^k$. For every directed $d\in Z$ and $x\in X^k$, either $d_x$ is finitary and in $\pi_k(G)$ or $d_x=d$. For every finitary $f\in Z$ and $x\in X^k$, $f_x=1$.  Set $O:=X^k$ and let $r_k:\Aut(\Xt)\rightarrow \Aut(O^*)$ be the canonical inclusion. Corollary~\ref{cor:preserve_BAG} ensures that $r_k(G)$ is again a bounded automata group, verifying (1). It is easy to verify that $\grp{Z}\leq \Aut(\Xt)$ is a bounded automata group using that $Z$ is closed under taking section, so $H=r_k(\grp{Z})$ is a bounded automata group by a second application of Corollary~\ref{cor:preserve_BAG}.

Let us now argue that $H\leq \Aut(O^*)$ is in reduced form. Set $Y:=r_k(Z)$; note that $Y$ is closed under takings sections by Lemma~\ref{lem:directed_preserve} as $Z$ is closed under taking sections.  Via Lemma~\ref{lem:directed_preserve}, $Y$ consists of finitary elements and directed elements, since $Z$ consists of such elements.  Say that $r_k(f)\in Y$ is finitary. Lemma~\ref{lem:directed_preserve} ensures that $f\in Z$ must also be finitary. For each $x\in O$, $f_x=1$, so in view of Lemma~\ref{lem:directed_preserve}, $(r_k(f))_x=1$ for all $x\in O$.  Say that $r_k(d)\in Y$ is directed. It follows that $d\in Z$ must be directed. For $x\in O$ such that $d_x$ is finitary, we have that $d_x\in Z$. Applying again Lemma~\ref{lem:directed_preserve}, we deduce that $(r_k(d))_x$ is finitary and an element of $\pi_1(H)$. For $x\in O$ such that $d_x=d$, we likewise deduce that $(r_k(d))_x=r_k(d)$. For all $x\in O$, it is thus the case that $(r_k(d))_x\in \{r_k(d)\}\cup \pi_1(H)$. Moreover, if $d$ is strongly directed, then so is $r_k(d)$ by Lemma~\ref{lem:directed_preserve}. We conclude that $H$ is in reduced form and that if each directed element of $Z$ is strongly directed, then $H$ is generalized basilica. Claim (2) thus holds.

For claim (3), each $g\in W$ is such that $g_x\in Z$ for any $x\in X^k$ by choice of $k$. From Lemma~\ref{lem:directed_preserve}, we deduce that for each $r_k(g)\in r_k(G)$ and $o\in O$, $(r_k(g))_o\in r_k(\grp{Z})=H$. We define $G\rightarrow \Sym(O)\ltimes H^O$ by $g\mapsto \pi_1(r_k(g))(r_k(g)_o)_{o\in O}$. One easily verifies that this is a well-defined monomorphism.
\end{proof}

Using a second fact from the literature, we obtain a useful corollary.
\begin{prop}[{\cite[Proposition 2.11.3]{N05}}] If $G$ is a self-replicating bounded automata group, then $Z$ as in Theorem~\ref{thm:gen_by_nuc} is a generating set for $G$.
\end{prop}

\begin{cor}\label{cor:reduction-reduced-form}
 If $G\leq \Aut(\Xt)$ is a self-replicating bounded automata group, then there is $k\geq 1$ such that $r_k(G)\leq\Aut((X^k)^*)$ is a self-replicating bounded automata group in reduced form. In particular, $G$ has a faithful representation as a self-replicating bounded automata group in reduced form.
\end{cor}

\begin{rmk}
Theorem~\ref{thm:reduction-reduced-form} motivates the form of our definitions of generalized basilica groups and abelian wreath type groups. Our definitions assume the group is in reduced form to avoid cumbersome reduction steps.
\end{rmk}

We deduce two further corollaries of Theorem~\ref{thm:reduction-reduced-form}, which elucidate the connection between torsion-free bounded automata groups and generalized basilica groups.  In a torsion-free bounded automata group, every finitary element is trivial and hence every directed element is strongly directed.  Theorem~\ref{thm:reduction-reduced-form} thus implies the following consequence.

\begin{cor}\label{cor:torsion-free_GBG}
If $G\leq \Aut(\Xt)$ is a torsion-free bounded automata group, then there are a torsion-free generalized basilica group $H$ such that $H$ is isomorphic to a subgroup of $G$ and $k\geq 1$ such that
\[
G\injects  \Sym(X^k)\ltimes H^{X^k}.
\]
\end{cor}

\begin{cor}\label{cor:torsion_free}
If $G\leq \Aut(\Xt)$ is a self-replicating torsion-free bounded automata group, then $G$ admits a faithful representation as a generalized basilica group.
\end{cor}

Let us conclude this section by working out the reduced form of the classical basilica group.
\begin{example}\label{ex:basilica}
Let $a,b\in \Aut([2]^*)$ be defined recursively by
\[
a(iv):=
\begin{cases}
iv & i=0\\
ib(v) & i=1
\end{cases}
\quad \text{and}\quad
b(iv):=
\begin{cases}
1v & i=0\\
0a(v) & i=1.
\end{cases}
\]
The group $G:=\grp{a,b}\leq \Aut([2]^*)$ is called the \textbf{basilica group}; see \cite{GZ02}. There is a convenient notation for the generators, representing them as elements of $\Sym([2])\ltimes G^2$, called wreath recursion: the wreath recursion of $a$ is $a=(1,b)$, and that of $b$ is $b=(01)(1,a)$.

We see that $a_1=b$ and $b_1=a$, so $G$ is not in reduced form. Consider $r_2(G)$. Set $\tilde{a}:=r_2(a)$ and $\tilde{b}:=r_2(b)$. In view of Lemma~\ref{lem:directed_preserve}, $\tilde{a}_{11}=r_2(a_{11})=r_2(a)=\tilde{a}$, and $\tilde{b}_{11}=\tilde{b}$. One further checks that all other sections of $\tilde{a}$ and $\tilde{b}$ on $[2]^2$ are trivial. In wreath recursion, $\at=(23)(1,1,1,\at)$ and $\bt=(02)(13)(1,1,1,\bt)$. Hence, $r_2(G)$ is in reduced form, and furthermore, the group is a generalized basilica group. In fact, $r_2(G)$ is a \textit{balanced} generalized basilica group.
\end{example}

\section{Technical results}
This section establishes the key technical result of this work, Lemma~\ref{lem:verbal_sgroup}, as well as several general technical observations for later use.

\subsection{Elementary amenable self-replicating groups}
Let $\gamma(x_1,\dots,x_n)$ be a word in the free group on $n$ generators. For a group $G$, the \textbf{verbal subgroup} of $G$ given by $\gamma(x_1,\dots,x_n)$ is
\[
\gamma(G):=\grp{\gamma(g_1,\dots,g_n)\mid g_i\in G}.
\]
Note that verbal subgroups are always characteristic subgroups of $G$.

A group is called \textbf{max-n} if every increasing chain of normal subgroups eventually stabilizes.  We will only require that all finitely generated virtually abelian groups are max-n, but the class of max-n groups is in fact large and complicated.

Recall that $\rk(G)$ denotes the construction rank of an elementary amenable group $G$.

\begin{lem}\label{lem:verbal_sgroup_v2}
	Suppose that $G\leq \Aut(\Xt)$ is an elementary amenable self-replicating group. Let $\gamma$ and $\beta$ be words such that $G/\beta(\gamma(G))$ is a max-n group.  Then there are weakly self-replicating normal subgroups $M$ and $ N$ of $G$ with
\begin{enumerate}[(1)]
\item $N\leq M$, $\gamma(G)\leq M$, and $\beta(\gamma(G))\leq N$;
\item $\rk(M)=\rk(\gamma(G))$ and $\rk (N)=\rk(\beta(\gamma(G)))$; and
\item $\beta(M/N)=1$.
\end{enumerate}
\end{lem}
\begin{proof}
	Fix $x\in X$ and let $\phi_x:G_{(x)}\rightarrow G$ be the section homomorphism at $x$. Clearly $\gamma(G_{(x)})\leq \gamma(G)_{(x)}\leq G_{(x)}$, and as $\phi_x$ is onto, we have
	\[
	\gamma(G)=\phi_x(\gamma(G_{(x)}))\leq \phi_x(\gamma(G)_{(x)})\leq \phi_x(G_{(x)})=G.
	\]
Similarly, we have
\[
	\beta(\gamma(G))=\phi_x(\beta(\gamma(G_{(x)})))\leq \phi_x(\beta(\gamma(G))_{(x)})\leq \phi_x(G_{(x)})=G.
	\]

	We now define two sequences of normal subgroups of $G$: Set $K_0:=\gamma(G)$, $H_0:=\beta(\gamma(G))$ and $K_{n+1}:=\phi_x((K_{n})_{(x)})$, $H_{n+1}:=\phi_x((H_n)_{(x)})$.  Trivially, $H_0\leq K_0$, so $H_n\leq K_n$ for all $n\geq 0$ by induction.   The previous paragraph ensures that $K_0\leq K_1$ and $H_0\leq H_1$, so $K_n\leq K_{n+1}$, $H_n\leq H_{n+1}$ for all $n\geq 0$, again by induction. Normality follows by induction because $\phi_x$ is onto. We claim that $\beta(K_n/H_n)=1$ for all $n\geq 0$. This is trivial for $n=0$. Assume that $\beta(K_n/H_n)=1$.  Note that $(K_n)_{(x)}\cap H_n = (H_n)_{(x)}$ and so $\beta((K_n)_{(x)}/(H_n)_{(x)})\leq \beta(K_n/H_n)=1$.  As $K_{n+1}/H_{n+1}$ is a quotient of $(K_n)_{(x)}/(H_n)_{(x)}$, we deduce that $\beta(K_{n+1}/H_{n+1})=1$.

Since $G/\beta(\gamma(G))$ is a max-n group, there is $L$ such that $K_L=K_{L+1}$ and $H_L=H_{L+1}$. We claim $M:=K_L$ and $N:=H_L$ for such an $L$ satisfy the lemma. The previous paragraph ensures that $M$ and $N$ are normal in $G$. We see that $\sec_M(x)=K_{L+1}=K_L=M$, and since the section subgroups are conjugate to each other, $\sec_M(y)=M$ for all $y\in X$. The obvious induction argument now gives that $\sec_M(v)=M$ for all $v\in \Xt$, so $M$ is weakly self-replicating. The same argument shows that $N$ is weakly self-replicating. Claims (1) and (3) follow from the previous paragraph. In view of Proposition~\ref{prop:rank_stability}, induction shows $\rk(K_j)= \rk(K_0)=\rk(\gamma(G))$ for all $j$. The group $M$ is thus such that $\rk(M)=\rk(\gamma(G))$. A similar argument shows $\rk(N)=\rk(\beta(\gamma(G)))$, completing the proof of (2).
\end{proof}

By considering the special case that $\beta$ is a letter, we obtain the following.

\begin{lem}\label{lem:verbal_sgroup}
	Suppose that $G\leq \Aut(\Xt)$ is an elementary amenable self-replicating group. Let $\gamma$ be a word such that $G/\gamma(G)$ is a max-n group.  Then there is a weakly self-replicating $M\normal G$ such that $\gamma(G)\leq M$ and $\rk(M)=\rk(\gamma(G))$.
\end{lem}

There are a couple of important corollaries of Lemma~\ref{lem:verbal_sgroup}. We say a word $w(x_1,\dots,x_n)$ is  \textbf{$d$-locally finite} if every $d$-generated group $G$ such that $w(G)=\{1\}$ is finite.

\begin{lem}[cf. {\cite[p. 517]{FJ08}}]\label{lem:lf_word}
	For every $d$-generated finite group $A$, there is a $d$-locally finite word $w$ such that $w(A)=1$.
\end{lem}

\begin{cor}\label{cor:ext_EA}
Suppose that $G\leq \Aut(\Xt)$ is a finitely generated elementary amenable self-replicating group. If $G$ is non-abelian, then there is a weakly self-replicating $M\normal G$ such that $G/M$ is either finite or abelian and $\rk(M)+1=\rk(G)$. Furthermore, if $\rk(G)$ is witnessed (in the sense of Corollary~\ref{cor:EA-rank}) by a normal subgroup $N$, then $G/M$ is finite if $G/N$ is finite and $G/M$ is abelian if $G/N$ is abelian.
\end{cor}

\begin{proof}
That $G$ is self-replicating ensures that $G$ is not finite. Since $G$ is non-abelian, Corollary~\ref{cor:EA-rank} supplies $N\normal G$ such that $G/N$ is either finite or abelian and $\rk(N)+1=\rk(G)$.

If $G/N$ is abelian, Lemma~\ref{lem:verbal_sgroup} applied to the verbal subgroup $[G,G]$ supplies the desired subgroup $M$. Let us suppose that $N$ is of finite index in $G$. Say that $G$ is $d$-generated and let $\gamma$ be a $d$-locally finite word such that $\gamma(A)=1$ where $A:=G/N$, whose existence is given by Lemma~\ref{lem:lf_word}. We thus have that $\gamma(G)\leq N$. It is also that case that $G/\gamma(G)$ is finite, since $G/\gamma(G)$ is a $d$-generated group and $\gamma(G/\gamma(G))=1$, so $\rk(\gamma(G))+1=\rk(G)$. Applying Lemma~\ref{lem:verbal_sgroup} to $\gamma(G)$ supplies the desired subgroup $M$.
\end{proof}

In the case of a self-replicating elementary amenable group for which the rank is given by a finite index subgroup, we can say a bit more.

\begin{cor}\label{cor:ext_EA_2}
Suppose that $G\leq \Aut(\Xt)$ is a finitely generated elementary amenable self-replicating group. If $G$ is not virtually abelian and admits $M\normal G$ with finite index such that $\rk(M)+1=\rk(G)$, then there are weakly self-replicating subgroups $A\normal G$ and $F\normal G$ such that
\begin{enumerate}
\item $A\leq F$, $F/A$ is abelian, and $G/F$ is finite; and
\item $\rk(F)+1=\rk(A)+2=\rk(G)$.
\end{enumerate}
\end{cor}
\begin{proof}
Since $M$ is finitely generated and neither finite nor abelian, as $G$ is not virtually abelian, Corollary~\ref{cor:EA-rank} supplies a normal subgroup $N\normal M$ such that $M/N$ is either finite or abelian and $\rk(N)+1=\rk(M)$.   It cannot be the case that $M/N$ is finite, since this contradicts the rank of $G$, so $M/N$ is abelian.  Say that $G$ is $d$-generated and let $\gamma$ be a $d$-locally finite word such that $\gamma(A)=1$ where $A:=G/M$, whose existence is given by Lemma~\ref{lem:lf_word}. It follows that $\gamma(G)\leq M$ and $G/\gamma(G)$ is finite.  Since $M/N$ is abelian, it follows that $[\gamma(G),\gamma(G)]\leq N$.  Thus
\[
\rk(G)\leq \rk(\gamma(G))+1\leq \rk(M)+1=\rk(G),
\]
and
\[
\rk(\gamma(G))\leq \rk([\gamma(G),\gamma(G)])+1\leq \rk(N)+1=\rk(M)=\rk(\gamma(G)).
\]
We conclude that  $\rk(\gamma(G))+1=\rk(G)$ and $\rk([\gamma(G),\gamma(G)])+2=\rk(\gamma(G))+1=\rk(G)$.

The group $G/[\gamma(G),\gamma(G)]$ is a finitely generated virtually abelian (as $G/\gamma(G)$ is finite) and hence a max-n group. Applying Lemma~\ref{lem:verbal_sgroup_v2} with $\beta=[x,y]$, we can find weakly self-replicating $F$ and $ A$ normal in $G$ such that $A\leq F$, $F/A$ is abelian, $\gamma(G)\leq F$ (and so $G/F$ is finite), $\rk(F)=\rk(\gamma(G))$ and $\rk(A)=\rk([\gamma(G),\gamma(G)])$.  This completes the proof.
\end{proof}

\subsection{Generalities on self-similar groups}
The results here are primarily for the reader's convenience, as these results are easy consequences of the definitions. We will appeal to these results throughout this work and often without explicit reference. The reader already comfortable with groups acting on rooted trees can safely skip this subsection.

\begin{lem}\label{lem:fix}
	Suppose that $G\leq \Aut(X^*)$, $G$ acts level transitively on $X^*$, and $H\normal G$ is weakly self-similar. If $H$ fixes $x\in X^k$ for some $k\geq 1$, then $H$ is trivial.
\end{lem}
\begin{proof}
That $H$ is normal ensures that $H$ fixes $X^k$. That $H$ contains $\sec_H(v)$ for all $v\in \Xt$ implies that $H$ in fact fixes $(X^k)^n$ for all $n\geq 1$. Hence, $H$ acts trivially on $\Xt$.
\end{proof}

\begin{lem}\label{lem:com_contain_sections_2}
Suppose that $G\leq \Aut(\Xt)$ acts transitively on $X$ and suppose that there is a weakly self-similar $M\normal G$ such that $G/M$ is abelian or $M$ acts transitively on $X$. If $g\in G$ fixes $y$ and $x$ in $X$ and $g_x=1$, then $g_y\in M$.
\end{lem}
\begin{proof}
Take $f\in G$ such that $f(x)= y$. If $M$ acts transitively on $X$, we take $f\in M$. The commutator $[f,g^{-1}]$ is an element of $M$, and $[f,g^{-1}]_y\in M$ since $[f,g^{-1}]$ fixes $y$ and $M$ is weakly self-replicating. Noting that $(f^{-1})_y=(f_x)^{-1}$, we see
\[
[f,g^{-1}]_y=(fg^{-1}f^{-1}g)_y=f_x\cdot 1 \cdot (f_x)^{-1}\cdot g_y=g_y.
\]
Hence, $g_y\in M$.
\end{proof}

\begin{lem}\label{lem:weakly_self-similar_trans}
Suppose that $G\leq \Aut(\Xt)$ acts transitively on $X$ and is weakly self-replicating. If $H\normal G$ is weakly self-similar and of finite index in $G$, then $H$ acts transitively on $X$.
\end{lem}
\begin{proof}
Fix $x\in X$ and let $\phi_x:G_{(x)}\rightarrow G$ be the section homomorphism, which is surjective because $G$ is weakly self-replicating. Since $H$ is weakly self-similar, the surjective homomorphism $\tilde{\phi}_x:G_{(x)}H/H\rightarrow G/H$ by $gH\mapsto \phi_x(g)H$ is well-defined. The group $G_{(x)}H/H$ is thus a subgroup of $G/H$ that surjects onto $G/H$. As $G/H$ is finite, we deduce that $G/H=G_{(x)}H/H$. Hence, $H$ acts transitively on $X$.
\end{proof}

\subsection{Example: the basilica group}\label{ex:basilica_1}
We close our technical discussion by exhibiting how these results are applied to prove the basilica group is non-elementary amenable. The basic strategy applied here is used throughout this work.

Let $G:=\grp{\at,\bt}\leq\Aut([4]^*)$ be the basilica group in the reduced form as given in Example~\ref{ex:basilica}. The group $G$ is self-replicating; one can verify this directly or see \cite{GZ02}. The elements $\at$ and $\bt$ in wreath recursion are $\at=(23)(1,1,1,\at)$ and $\bt=(02)(13)(1,1,1,\bt)$.

Let us suppose toward a contradiction that $G$ is elementary amenable. One easily verifies that $G$ is not abelian. Corollary~\ref{cor:ext_EA} thus produces a weakly self-replicating subgroup $M\normal G$ such that $\rk(M)+1=\rk(G)$ and $G/M$ is either finite or abelian. We see that $\at^2=(1,1,\at,\at)$ and $\bt^2=(1,\bt,1,\bt)$. Lemmas~\ref{lem:com_contain_sections_2} and \ref{lem:weakly_self-similar_trans} imply that $\at\in M$ and $\bt\in M$, so $M=G$. This contradicts the rank of $G$. We conclude that $G$ is not elementary amenable.

\begin{rmk}
The general strategy to show a given self-replicating bounded automata group is not elementary amenable is to contradict the rank of the group via the weakly self-replicating subgroups provided by either Corollary~\ref{cor:ext_EA} or Corollary~\ref{cor:ext_EA_2}. In the case of the basilica group, the squares $\at^2$ and $\bt^2$ act trivially on $[4]$ and have some trivial sections, so we may apply Lemmas~\ref{lem:com_contain_sections_2} and \ref{lem:weakly_self-similar_trans} to deduce that $M=G$, which contradicts the rank. In general, we must work much harder.
\end{rmk}

\section{Groups containing odometers}
It is convenient to begin with our theorems on bounded automata groups containing odometers.
\begin{defn}
An element $d\in \Aut(\Xt)$ is said to be an \textbf{odometer} if $\grp{d}$ acts transitively on $X$ and $d_x\in \{1,d\}$ for all $x\in X$ with exactly one $x\in X$ such that $d_x=d$.
\end{defn}
In the course of proving these results, we upgrade Corollary~\ref{cor:ext_EA_2}, and the resulting more powerful statement, Corollary~\ref{cor:ext_EA_3}, will be used frequently in later sections.

\subsection{Self-replicating groups}
This preliminary section shows that bounded automata groups containing an odometer are close to being self-replicating. We begin by characterizing self-replication for generalized basilica groups.

\begin{lem}\label{lem:trans_no_section}
Suppose that $h\in \Aut(\Xt)$ is such that $h_x\in \{1,h\}$ for all $x\in X$ with at most one $x$ such that $h_x=h$. For any $w,v\in X$ and $j\in \Zb$ for which $h^j(v)=w$, there is $i$ such that $h^i(v)=w$ and $(h^i)_v=1$ and $0\leq |i|<|X|$.
\end{lem}
\begin{proof}
If no $x\in X$ is such that $h_x=h$, then the lemma is immediate.

Suppose there is a unique $x\in X$ with $h_{x}=h$. Choose $j$ with $|j|\geq 0$ least such that $h^j(v)=w$.  If $(h^j)_v=1$, we are done. Otherwise, it is the case that $(h^j)_v=h^{\pm 1}$, since we chose $|j|$ least. As that cases are similar, let us suppose that $(h^j)_v=h$ (i.e., $j>0$).

Let $O$ be the orbit of $x$ under $\grp{h}$. The word $x$ is the unique element of $X$ for which $h$ has a non-trivial section, so $(h^k)_z=1$ for all $z\in X\setminus O$ and $k\in \Zb$. We deduce that $v\in O$. Take $m\geq 1$ least such that $h^m$ fixes $x$ and observe that $(h^m)_y=h$ for all $y\in O$. The element $h^{-m}h^j=h^{j-m}$ is thus such that $h^{j-m}(v)=w$ and $(h^{j-m})_v= h^{-1}h=1$ and $|j-m|<|X|$.
\end{proof}

\begin{lem}\label{lem:section_sgrp}
	Suppose that $G=\grp{Y}\leq \Aut(\Xt)$ is a generalized basilica group and let $O$ be an orbit of $G$ on $X$. For all $x\in O$, it is then the case that $\sec_G(x)=\grp{Y'}$ where $Y'$ is the collection of directed $g\in Y$ such that $v^g\in O$.
\end{lem}
\begin{proof} Take $d\in Y$ and let $W_1,\dots, W_l$ list the orbits of $\grp{d}$ on $O$, for some $1\leq k \leq n$. For $w,v\in W_j$, there is some power $i$ such that $d^i(v)=w$ and $(d^i)_v=1$ by Lemma~\ref{lem:trans_no_section}. Plainly, $d^{i}G_{(v)}d^{-i}=G_{(w)}$, and we deduce that
	\[
	\sec_G(w)=(d^{i})_v\sec_G(v)((d^i)_v)^{-1}.
	\]
 Hence, $\sec_G(w)=\sec_G(v)$, 	since $(d^{i})_v=1$. We conclude that $\sec_G(w)=\sec_G(v)$ for all $w,v\in W_j$ and $1\leq j\leq l$. Since $G=\grp{Y}$ and acts transitively on $O$, it follows that $\sec_G(v)=\sec_G(w)$ for all $v,w\in O$.
	
	For $x\in O$ and $h\in Y$, the section $h_x$ is an element of $\grp{Y'}$. Inducting on the word length, one sees that $g_x\in \grp{Y'}$ for all $g\in G$, so a fortiori $\sec_G(x)\leq \grp{Y'}$. On the other hand, for $d\in Y'$, take $n\geq 1$ least such that $d^n$ fixes $v^d$. The section of $d^n$ at $v^d$ is $d$, so $d\in \sec_G(v^d)=\sec_G(x)$. We conclude that $\grp{Y'}=\sec_G(x)$ for all $x\in O$, verifying the lemma.
\end{proof}

\begin{prop}\label{prop:self-rep_char}
For $G=\grp{Y}\leq \Aut(\Xt)$ a generalized basilica group, $G$ is self-replicating if and only if $G$ is generated by the directed elements of $Y$ and $G$ acts transitively on $X$.
\end{prop}
\begin{proof}
Suppose first that $G$ is self-replicating. By definition, $G$ acts transitively on $X$. Fix $x\in X$ and let $Z\subseteq Y$ be the directed elements. An easy induction argument shows that $g_x\in \grp{Z}$ for all $g\in G$. We deduce that $G=\sec_G(x)\leq \grp{Z}\leq G$, and thus, $G$ is generated by $Z$.

For the converse, let $Z$ be the directed elements of $Y$. Since $G$ acts transitively on $X$, $\sec_G(x)=\grp{Z}$ for all $x\in X$ by Lemma~\ref{lem:section_sgrp}. As $\grp{Z}=G$, we deduce that $\sec_G(x)=G$.
\end{proof}

A version of Proposition~\ref{prop:self-rep_char} holds for bounded automata groups with odometers.
\begin{prop}\label{prop:self-rep_char_gen}
Suppose that $G=\grp{Y}\leq \Aut(\Xt)$ is a bounded automata group in reduced form. If $G$ is self-replicating, then $G$ is generated by
\[
Z:=\{d_x\mid d\in Y\text{ is directed and }x\in X\}.
\]
The converse holds if $G$ contains an odometer.
\end{prop}
\begin{proof}
Suppose first that $G$ is self-replicating. Fix $x\in X$. An easy induction argument shows that $g_x\in \grp{Z}$ for all $g\in G$. We deduce that $G=\sec_G(x)\leq \grp{Z}\leq G$, and thus, $G$ is generated by $Z$.

For the converse, let $h\in G$ be an odometer. Immediately, we see that $G$ acts transitively on $X$. For each $x$ and $y\in X$ distinct, Lemma~\ref{lem:trans_no_section} supplies $i$ such that $h^i(x)=y$ and $(h^i)_x=1$. Fixing $x\in X$, for any $g\in G$  and $y\in X$, there are $i$ and $j$ such that $h^jgh^i$ fixes $x$ and $(h^jah^i)_x=g_y$. The image of the section homomorphism $\sec_G(x)=\phi_x(G_{(x)})$ therefore contains all sections $g_y$ for $g\in G$ and $y\in X$. In particular, $\sec_G(x)$ contains $Z$, so $\sec_G(x)=G$. We conclude that $G$ is self-replicating.
\end{proof}

Our final preliminary lemma shows that one can easily reduce to the group generated by the set $Z$.

\begin{lem}\label{lem:reduce_to_directed}
	Suppose that $G=\grp{Y}\leq \Aut(\Xt)$ is a bounded automata group in reduced form. Letting $Z:=\{d_x\mid d\in Y\text{ is directed and }x\in X\}$, the group $\grp{Z}$ is a bounded automata group in reduced form, and
	\[
	G\hookrightarrow \Sym(X)\ltimes \grp{Z}^X .
	\]
\end{lem}
\begin{proof} That $\grp{Z}$ is a bounded automata group in reduced form is immediate.

For each $x\in X$ and $h\in Y$, it follows the section $h_x$ is an element of $Z$. Inducting on the word length, one sees that $g_x\in \grp{Z}$ for all $g\in G$ and $x\in X$. The map
\[
	G\rightarrow\Sym(X)\ltimes \grp{Z}^X
\]
given	by $g\mapsto \pi_1(g)((g_x)_{x\in X})$ is thus well-defined. One verifies that this map is also a monomorphism.
\end{proof}

\subsection{Generalized basilica groups}

\begin{lem}\label{lem:fin_abelianization}
Suppose that $G=\grp{Y}\leq \Aut(\Xt)$ is a self-replicating generalized basilica group that is non-abelian. If $H\normal G$ is weakly self-replicating and $G/H$ is abelian, then $G/H$ is finite.
\end{lem}
\begin{proof}
In view of Proposition~\ref{prop:self-rep_char}, we may assume that $Y$ consists of directed elements.

Suppose first that $k\in Y$ does not act transitively on $X$ and let $n\geq 1$ be least such that $k^n$ fixes $X$. There is some $x\in X$ such that $(k^n)_x=1$, and there is $y\in X$ such that $(k^n)_y=k^i$ for some $i\geq 1$. Applying Lemma~\ref{lem:com_contain_sections_2}, we conclude that $(k^n)_y=k^i\in H$. Hence, $k$ is torsion in $G/H$.

We next consider $h\in Y$ that acts transitively on $X$. Suppose toward a contradiction that $h^i\notin H$ for all $i\geq 1$. Fix $k\in Y\setminus\{h^{-1}\}$ and fix $y\in X$ such that $k(y)\neq y$ and $k_y=1$. By Lemma~\ref{lem:trans_no_section}, there is $|X|>|i|\geq 1$ such that $h^ik(y)=y$ and $(h^i)_{k(y)}=1$. The element $h^ik$ has non-trivial sections exactly at some $x_1,\dots,x_l$ in $X$ where $l\in\{|i|,|i|+1\}$, since $h^i$ has $|i|$ many non-trivial sections (all equal to $h$ or $h^{-1}$) and $k$ has one non-trivial section. Say that $x_1$ is the element of $X$ at which the section of $h^ik$ is of the form $ak$ for $a\in \{1,h^{\pm 1}\}$.

Suppose there is some $x_j$ such that $x_j$ and $x_1$ lie in different orbits of the cyclic group $\grp{h^ik}$ acting on $X$. Let $m\geq 1$ be least such that $(h^ik)^m$ fixes $x_j$. The section $((h^ik)^m)_{x_j}$ is of the form $h^r$ for some $1\leq |r|\leq i$. The element $(h^ik)^m$ also fixes $y$ and has a trivial section at $y$. Applying Lemma~\ref{lem:com_contain_sections_2}, we deduce that $h^r\in H$, which is absurd. It is thus the case that all $x_j$ lie in the same orbit of  $\grp{h^ik}$ acting on $X$. Take $m\geq 1$ least such that $(h^ik)^m$ fixes $x_1$ and observe that $((h^ik)^m)_{x_1}=h^ik$. Applying again Lemma~\ref{lem:com_contain_sections_2}, we conclude that $h^ik\in H$.

For any other $y'\in X$ such that $k(y')\neq y'$ and $k_{y'}=1$, the same argument gives $|X|> |j|\geq 1$ such that $h^{j}k(y')=y'$ and $h^jk\in H$.  Thus, $h^{j-i}=h^{j}kk^{-1}h^{-i}\in H$. In view of our reductio assumption, we conclude that $i=j$. Therefore, $h^ik$ fixes all $y\in X$ such that $k(y)\neq y$ and $k_y=1$.

Say that $k$ fixes $x$ and $k_x=1$. Letting $l\geq 1$ be least such that $k^l$ fixes $v^k$,  Lemma~\ref{lem:com_contain_sections_2} implies that $k\in H$. Since $h^ik$ is also in $H$, $h^i\in H$, which is absurd. We conclude that $k$ fixes no $x$ such that $k_x=1$. The element $h^ik$ must then fix all $y\in X\setminus\{v^k\}$, so $h^ik$ fixes $X$. The subgroup $H$ contains all sections of $h^ik$, since it is weakly self-similar, so if $i>1$ or $i=1$ and $v_h\neq k(v_k)$, then $h\in H$, which we assume to be false. The element $hk$ therefore fixes $X$, and $(hk)_{v^k}=hk$. It follows that $hk=1$. This concludes the reductio argument, since $k\neq h^{-1}$.

 We have now established that every element $k\in Y$ is such that $k^i\in H$ for some $i$, hence $G/H$ is finite.
\end{proof}

\begin{thm}\label{thm:trans}
Suppose that $G=\grp{Y}\leq \Aut(\Xt)$ is a generalized basilica group. If $G$ contains an odometer, then either $G$ is virtually abelian or $G$ is not elementary amenable.
\end{thm}
\begin{proof}
Without loss of generality, we may assume that $Y$ contains $\pi_1(G)\cap G$.  Let $K$ be the subgroup generated by the directed members of $Y$. The subgroup $K$ equals $\grp{Z}$ where $Z$ is as in Lemma~\ref{lem:reduce_to_directed}, so in view of Lemma~\ref{lem:reduce_to_directed}, $K$ is virtually abelian if and only if $G$ is virtually abelian. By replacing $G$ with $K$ if needed, we assume that each element of $Y$ is directed. Proposition~\ref{prop:self-rep_char} now ensures that $G$ is self-replicating.

Let us suppose toward a contradiction that $G$ is elementary amenable but not virtually abelian. Applying Corollary~\ref{cor:ext_EA}, there is $H\normal G$ such that $H$ is weakly self-replicating, $\rk(H)+1=\rk(G)$, and $G/H$ is either finite or abelian. In the case that $G/H$ is abelian, Lemma~\ref{lem:fin_abelianization} ensures that $G/H$ is in fact finite. The group $G$ thus admits a weakly self-replicating $H\normal G$ such that $\rk(H)+1=\rk(G)$ and $G/H$ is finite. Since $H\normal G$ is of finite index, Corollary~\ref{cor:ext_EA_2} supplies a weakly self-replicating $A\normal G$ such that $G/A$ is virtually abelian and $\rk(A)+2=\rk(G)$.

Fix $h\in G$ such that $h$ is an odometer and form $L:=A\grp{h}$. For each $x$ and $y\in X$ distinct, Lemma~\ref{lem:trans_no_section} supplies $i$ such that $h^i(x)=y$ and $(h^i)_x=1$. Fixing $x\in X$, for any $a\in A$  and $y\in X$, there are $i$ and $j$ such that $h^jah^i$ fixes $x$ and $(h^jah^i)_x=a_y$. The image of the section homomorphism $\sec_L(x)=\phi_x(L_{(x)})$ therefore contains all sections of $A$.

For any $k\in Y$, there is some $a\in A$ such that $a(k(v^k))\neq k(v^k)$, by Lemma~\ref{lem:fix}. The element $k^{-1}ak$ is a member of $A$, and $(k^{-1}ak)_{v^k}=a_{k(v^k)}k$. The group $\sec_L(x)$ thus contains $a_{k(v^k)}k$ and $a_{k(v^k)}$. Hence, $k\in \sec_L(x)$, and we deduce that $\sec_L(x)=G$. On the other hand, $A_{(x)}\normal L_{(x)}$, and $L_{(x)}/A_{(x)}$ is abelian. Thus, $A=\sec_A(x)\normal \sec_L(x)=G$, so $G/A$ is abelian. This implies that $\rk(G)=\rk(A)+1$ which is absurd.
\end{proof}

\begin{example}\label{ex:basilica_2}
Let $G:=\grp{\at,\bt}\leq\Aut([4]^*)$ be the basilica group in the reduced form as given in Example~\ref{ex:basilica}. The elements $\at$ and $\bt$ in wreath recursion have the following form: $\at=(23)(1,1,1,\at)$ and $\bt=(02)(13)(1,1,1,\bt)$. We now see that
\[
\at\bt^{-1}=(23)(1,1,1,\at)(02)(13)(1,\bt^{-1},1,1)=(0312)(1,\at\bt^{-1},1,1).
\]
Hence, $\at\bt^{-1}$ is an odometer, so $G$ contains an odometer. Theorem~\ref{thm:trans} now gives a second proof that $G$ is not elementary amenable.
\end{example}

\subsection{Technical results revisited}

\begin{lem}\label{lem:BAG-rank}
Suppose that $G=\grp{Y}\leq \Aut(\Xt)$ is a self-replicating bounded automata group in reduced form. If $G$ is elementary amenable but not virtually abelian, then there is a weakly self-replicating $M\normal G$ such that $G/M$ is finite and $\rk(M)+1=\rk(G)$
\end{lem}
\begin{proof}
In view of Lemma~\ref{lem:directed_preserve}, if $G$ is in reduced form, then $r_k(G)\leq\Aut((X^k)^*)$ is again in reduced form for any $k\geq 1$. By passing to $r_2(G)$, we may thus assume that $|X|>2$.

By Corollary~\ref{cor:ext_EA}, there is a weakly self-replicating $N\normal G$ such that $G/N$ is either finite or abelian and $\rk(N)+1=\rk(G)$. Suppose that $G/N$ is abelian. We will argue that $G/N$ is finite. Letting $F$ be the collection of finitary elements of $G$, we will in fact show that $FN$ is of finite index in $G$. This suffices to prove the theorem as $F$ is locally finite, and hence $G/N$ will have a locally finite subgroup of finite index, and hence is finite.

Fix $d\in Y$ directed and suppose first that $\grp{d}$ does not act transitively on $X$. Let $x\in X$ be such that $d_x=d$ and $y\in X$ be such that $y$ is not in the orbit of $x$ under $\grp{d}$. Taking $m\geq 1$ to be least such that $d^m$ fixes $X$, it follows that $(d^m)_x=(fd)^i$ for some $i\geq 1$ and $f$ finitary and $(d^m)_y=r$ for $r$ some finitary. Since $G$ is self-replicating, there is $h\in G$ such that $h(y)=x$ and $h_y=1$. The commutator $[h,d^{-m}]$ is an element of $N$, and
\[
[h,d^{-m}]_x=r^{-1}(fd)^{i},
\]
which is an element of $N$ as $N$ is weakly self-replicating. Since $G/N$ is abelian, it follows that $d^i\in FN$.

Now suppose that $d\in Y$ acts transitively on $X$. In view of Theorem~\ref{thm:trans}, there must be a non-trivial finitary element $f\in Y$. We may find a power $n:=|X|>i\geq 1$ such that $d^if$ fixes some $x\in X$. This presents two cases: $(d^if)_x$ is finitary or $(d^if)_x=rdr'$ for $r$ and $r'$ finitary.  The cases are similar, so we only consider the latter. Recall that $n>2$. If $i=n-1$, then $d^{-1}f$ fixes $x$, so we replace $d$ with $d^{-1}$. We may thus assume that $n-1>i$.

Replace $d^i$ with $d^{j}$ where $j:=i-n$ and note that $1<|j|$. The sections of $d^jf$ have three possible forms: finitary, $rdr'd^{-1}r''$ with $r,r'$, and $r''$ finitary, or $rd^{- 1}r'$ with $r$ and $r'$ finitary. The section $(d^jf)_x$ is of the form $rdr'd^{-1}r''$, and since $i<n-1$, there is some $y\in X$ such that $(d^jf)_y=sd^{-1}s'$ with $s$ and $s'$ finitary elements.

Let $k$ be least such that $(d^jf)^k$ fixes $y$. Since $G$ is self-replicating, there is $h\in G$ such that $h(x)=y$ and $h_x=1$. The commutator $[(d^jf)^{k},h]$ is an element of $N$, and
\[
[(d^jf)^k,h]_y=a_k\dots a_1((d^jf)_x)^{-k},
\]
where each $a_i$ has one of aforementioned forms, $(d^jf)_x$ has the form $rdr'd^{-1}r''$ and $a_1= ((d^jf)^k)_y$ is of the form $sd^{-1}s'$.

Since $G/N$ is abelian, the $a_i$ of the form $rdr'd^{-1}r''$ are equal to a finitary modulo $N$. We may thus write $[(d^jf)^k,h]_y \mod N=b_1\dots b_m \mod N$ where each $b_i$ is either finitary or of the form $sd^{-1}s'$. Commuting the finitaries to the left modulo $N$, we see that $[(d^jf)^k,h]_y \mod N= f'd^{-l}\mod N$ for $f$ a finitary and $l$ non-zero. On the other hand, $[(d^jf)^k,h]_y \in N$ since $N$ is weakly self-replicating. It now follows that $d^l\in FN$.

We have now demonstrated that every $z\in Y$ admits $l$ such that $z^l\in FN$. Hence, $FN$ is of finite index in $G$ as $G/N$ is finitely generated abelian.
\end{proof}

In view of Corollary~\ref{cor:ext_EA_2}, Lemma~\ref{lem:BAG-rank} yields the following consequence.

\begin{cor}\label{cor:ext_EA_3}
Suppose that $G\leq \Aut(\Xt)$ is a self-replicating bounded automata group. If $G$ is elementary amenable but not virtually abelian, then there are weakly self-replicating subgroups $N\normal G$ and $M\normal G$ such that
\begin{enumerate}
\item $N\leq M$, $M/N$ is abelian, and $G/M$ is finite; and
\item $\rk(M)+1=\rk(N)+2=\rk(G)$.
\end{enumerate}
\end{cor}

Corollary~\ref{cor:ext_EA_3} ensures that various examples of elementary amenable groups do not have faithful representations as self-replicating bounded automata groups.

\begin{prop}\label{prop:no_rep}
Suppose that $G$ is an elementary amenable group that is not virtually abelian. If every finite index subgroup has the same rank as $G$, then $G$ has no faithful representation as a self-replicating bounded automata group.
\end{prop}
\begin{proof}
Suppose for contradiction $G$ has a faithful representation as a self-replicating bounded automata group. Appealing to Corollary~\ref{cor:reduction-reduced-form}, we may assume that $G\leq\Aut(\Xt)$ is a representation of $G$ as a bounded automata group in reduced form. Lemma~\ref{lem:BAG-rank} now implies that $G$ has a finite index subgroup of lower rank. This contradicts our hypotheses.
\end{proof}

\begin{cor}\label{cor:lamplighter}
The groups $A\wr \Zb$ where $A$ is a non-trivial abelian group do not have faithful representations as self-replicating bounded automata groups.
\end{cor}
Corollary~\ref{cor:lamplighter} is rather interesting because the classical lamplighter group $\Zb/2\Zb\wr \Zb$ does have representations as a self-replicating group \cite[cf. Proposition 1.9.1]{N05} (see also~\cite{SS05} where $\Zb/2\Zb$ is replaced by an arbitrary finite abelian group), and on the other hand, Example~\ref{ex:lamp} below shows it has representations as a bounded automata group. However, one can never find a representation which is simultaneously as a self-replicating group and as a bounded automata group

\subsection{The general case}
\begin{thm}\label{thm:odometers}
Suppose that $G\leq \Aut(\Xt)$ is a bounded automata group. If $G$ contains an odometer, then either $G$ is virtually abelian or $G$ is not elementary amenable.
\end{thm}
\begin{proof}
Suppose toward a contradiction that $G$ is elementary amenable but not virtually abelian. Via Theorem~\ref{thm:reduction-reduced-form}, it is enough to consider $G=\grp{Y}\leq \Aut(\Xt)$ to be in reduced form. Without loss of generality, we may assume that $Y$ contains an odometer $h$.

Setting $Z:=\{d_x\mid d\in Y\text{ is directed and }x\in X\}$, Lemma~\ref{lem:reduce_to_directed} implies that $H:=\grp{Z}$ is a bounded automata group in reduced form, and $G\injects \Sym(X)\ltimes H^{X}$. Moreover, $h\in Z$, so $H$ contains an odometer and is elementary amenable but not virtually abelian. Proposition~\ref{prop:self-rep_char_gen} further implies that $H$ is self-replicating.

Applying Corollary~\ref{cor:ext_EA_3}, there is a weakly self-replicating $M\normal H$ such that $H/M$ is virtually abelian and $\rk(M)+2=\rk(H)$. Using the odometer $h\in H$, set $L:=M\grp{h}$. For each $u$ and $w\in X^2$ distinct, Lemma~\ref{lem:trans_no_section} supplies $i$ such that $h^i(u)=w$ and $(h^i)_u=1$. Fixing $v\in X^2$, for any $m\in M$  and $w\in X^2$, there are $i$ and $j$ such that $h^jmh^i$ fixes $v$ and $(h^jmh^i)_v=m_w$. The image of the section homomorphism $\sec_L(v)=\phi_v(L_{(v)})$ therefore contains $m_w$ for every $m\in M$ and $w\in X^2$.

For any directed $k\in Y$ and $x\in X$ the active vertex of $k$, there is some $a\in M$ such that $ak(x)\neq k(x)$, by Lemma~\ref{lem:fix}. The section $(k^{-1})_{ak(x)}$ is finitary of depth one, so
\[
(k^{-1})_{ak(x)x'}=((k^{-1})_{ak(x)})_{x'}=1
\]
for any $x'\in X$. The element $k^{-1}ak$ is a member of $M$, and for any $y\in X$,
\[
(k^{-1}ak)_{xy}=(k^{-1})_{ak(x)a_{k(x)}(k(y))}a_{k(x)k(y)}k_y=a_{k(x)k(y)}k_y
\]
 The group $\sec_L(v)$ thus contains $a_{k(x)k(y)}k_y$ and $a_{k(x)k(y)}$. Hence, $k_y\in \sec_L(v)$. We deduce that $Z\subseteq \sec_L(v)$, so $\sec_L(v)=H$. On the other hand, $M_{(v)}\normal L_{(v)}$, and $L_{(v)}/M_{(v)}$ is abelian. Since $M=\sec_M(v)\normal \sec_L(v)=H$, we deduce that $H/M$ is abelian. Therefore, $\rk(H)=\rk(M)+1$ which is absurd.
\end{proof}

\section{Kneading automata groups}
\subsection{Preliminaries}

For a permutation $\sigma\in \Sym(X)$, a \textbf{complete cycle decomposition} of $\sigma$ is a cycle decomposition including all length one cycles.
\begin{defn}
Let $X$ be a finite set and $\alpha=(\sigma_1,\dots,\sigma_n)$ be a sequence of permutations of $X$. Say that each $\sigma_i$ has complete cycle decomposition $c_{i,1}\dots c_{i,k_i}$. We define the \textbf{cycle graph} of $\alpha$, denoted by $\Gamma_{\alpha}$, as follows:
\[
V\Gamma_{\alpha}:=X\sqcup \{(i,j)\mid 1\leq i\leq n \text{ and } 1\leq j \leq k_i\}
\]
and
\[
E\Gamma_{\alpha}:=\left\{\{(i,j),x\}\mid x \text{ appears in }c_{i,j}\right\}.
\]

\end{defn}
The graph $\Gamma$ is a bipartite graph with bipartition consisting of $X$ and $\mc{C}:= \{(i,j)\mid 1\leq i\leq n \text{ and } 1\leq j \leq k_i\}$.  Some authors define the cycle graph using only non-trivial cycles. We allow trivial cycles to make our discussion more streamlined later. Allowing for trivial cycles only adds leaves to the graph. In particular, these leaves to do affect whether or not the cycle graph is a tree.

\begin{defn}
For $X$ a finite set, a finite sequence $\alpha$ of elements of $\Sym(X)$ is called \textbf{tree-like} if the cycle graph is a tree.
\end{defn}

We will find tree-like sequences in the definition of kneading automata groups. Let us note several facts for later use.

\begin{lem}
If $(\sigma_1,\dots,\sigma_n)$ is a tree-like sequence of permutations of a finite set $X$, then $H:=\grp{\sigma_1,\dots,\sigma_n}$ acts transitively on $X$.
\end{lem}
\begin{proof}
First note that if $x,x'\in X$ are connected by a path of length $2$, then they belong to a cycle of some $\sigma_i$ and are hence in the same orbit of $H$.  Since the cycle graph $\Gamma$ is bipartite, and hence any path between elements of $X$ decomposes as a composition of length two paths of the above sort, $H$ is transitive by the connectivity of $\Gamma$.
\end{proof}

For a permutation $\sigma\in \Sym(X)$, we denote the collection of fixed points in $X$ of $\sigma$ by $\fix(\sigma)$.

\begin{lem}[{\cite[Lemma 6.5]{J15}}]\label{lem:fixed_pts_tree-like}
Let $X$ be a finite set and $\alpha=(\sigma_1,\dots,\sigma_n)$ be a tree-like sequence of permutations of $X$. Then
\begin{enumerate}
\item For any $i\neq j$, $|\fix(\sigma_i)|+|\fix(\sigma_j)|\geq 2$.
\item If $|\fix(\sigma_i)|+|\fix(\sigma_j)|\leq 3$ for some $i\neq j$, then $\sigma_k=1$ for all $k\notin\{i,j\}$.
\end{enumerate}
\end{lem}

\subsection{Kneading automata}
A set $Y\subseteq \Aut(\Xt)$ is called \textbf{self-similar} if $g_x\in Y$ for every $x\in X$ and $g\in Y$.

\begin{defn}
We say that $G\leq\Aut(\Xt)$ is a \textbf{kneading automata group} if $G$ admits a finite self-similar generating set $Y$ such that the following hold:
\begin{enumerate}
\item For each non-trivial $h\in Y$, there is a unique $g\in Y$ and $x\in X$ such that $g_x=h$.
\item For each $h\in Y$ and each cycle $(x_1\dots x_n)$ of $\pi_1(h)$, possibly of length one, there is at most one $1\leq i\leq n$ such that $h_{x_i}\neq 1$.
\item The sequence $(\pi_1(h))_{h\in Y\setminus\{1\}}$ is tree-like.
\end{enumerate}
We call $Y$ a \textbf{distinguished generating set} and write $G=\grp{Y}\leq\Aut(\Xt)$ to indicate that $Y$ is a distinguished generating set for $G$.
\end{defn}

A straightforward verification shows that a kneading automata group is a bounded automata group. Furthermore, kneading automata groups naturally arise in the study of iterated monodromy groups.

\begin{thm}[{Nekrashevych, \cite[Theorem 6.10.8]{N05}}]\label{thm:iterated_monodromy}
Every iterated monodromy group of a post-critically finite polynomial has a faithful representation as a kneading automata group
\end{thm}

Kneading automata groups enjoy several useful properties.
\begin{lem}\label{lem:kneading_BAG}
If $G=\grp{Y}\leq \Aut(\Xt)$ is a kneading automata group, then $G\leq \Aut(\Xt)$ is a self-replicating bounded automata group
\end{lem}
\begin{proof}
To see that $G$ is self-replicating, we must verify that $G$ acts transitively on $X$ and the section homomorphisms are onto for every $x\in X$. Since $(\pi_1(y))_{g\in Y\setminus\{1\}}$ is tree-like, $\pi_1(G)$ acts transitively on $X$, so $G$ acts on $X$ transitively.

To see that the section homomorphisms are onto, we first argue that for each $v,w\in X$, there is $k\in G$ such that $k(v)=w$ with $k_v=1$. Fix $g\in Y$ and say that $\{O_1,\dots,O_n\} $ lists the orbits of $\grp{g}$ on $X$. Take $v,w\in O_i$ and fix $|j|\geq 1$ least such that $g^j(v)=w$. If $(g^j)_v=1$, we are done, so let us suppose that $(g^j)_v$ is non-trivial.  By the definition of a kneading automata group, there is a most one $z\in O_i$ such that $g_{z}$ is non-trivial. Say that $g_{z}=h$ and let $m\geq 1$ be least such that $g^m$ fixes $z$. The element $g^m$ fixes $O_i$ and $(g^m)_u=h$ for every $u\in O_i$. On the other hand, $(g^j)_v=h^{\pm 1}$, since we chose $j$ with $|j|$ least. Either $g^{j-m}$ or $g^{j+m}$, depending on the sign of $(g^j)_v$, then sends $v$ to $w$ with trivial section at $v$.

For an arbitrary $v,w\in X$, there is a word $\gamma=g_n^{i_n}\dots g_{1}^{i_1}$ with $g_k\neq g_{k+1}\in Y$ such that $\gamma(v)=w$. We may further assume that $\Sigma_{j=1}^n|i_j|$ is least among all such words $\gamma$. The image $g_1^{i_1}(v)$ lies in the orbit of $v$ under $\grp{g_1}$. Appealing to the previous paragraph, we may replace $g_1^{i_1}$, if necessary, with $g_1^{i_1\pm m_i}$ for some $m_i$ such that that $g_1^{i_1\pm m_i}(v)=g_1^{i_1}(v)$ and $(g_1^{i_1\pm m_i})_v=1$. Doing this for each $g_j$ as necessary, we produce a new word $\gamma'$ such that $\gamma'(v)=w$ and $(\gamma')_v=1$.

Fix $x\in X$ and let $g\in Y$. By the definition of a kneading automata group, there is $h\in Y$ and $z\in X$ such that $h_z=g$. The previous paragraph supplies $\gamma$ and $\delta$ such that $\gamma(x)=z$ with $\gamma_x=1$ and $\delta(h(z))=x$ with $\delta_{h(z)}=1$. The element $\delta h \gamma$ fixes $x$, and
\[
(\delta h \gamma )_x =\delta_{h(z)}h_z\gamma_x=h_z=g.
\]
The image $\phi_x(G_{(x)})$ of the section homomorphism at $x$ thus contains $Y$. We conclude that $\phi_x(G_{(x)})=G$, and $G$ is self-replicating.
\end{proof}

In addition to self-replication, kneading automata groups enjoy a robustness property. Specifically, the representation $r_k(G)$ is again a kneading automata group, for any $k\geq 1$.  This was proved in~\cite[Proposition 6.7.5]{N05} using the dual automaton and a topological argument; here we provide a simple counting argument.

\begin{lem}[{\cite[Proposition 6.7.5]{N05}}]\label{lem:kneading_reduced_form}
If $G=\grp{Y}\leq \Aut(\Xt)$ is a kneading automata group, then
\begin{enumerate}
\item for all $k\geq 1$, $r_k(G)=\grp{r_k(Y)}\leq\Aut((X^k)^*)$ is a kneading automata group;

\item there is $k\geq 1$ such that $r_k(G)=\grp{r_k(Y)}\leq\Aut((X^k)^*)$ is a kneading automata group in reduced form.
\end{enumerate}
In particular, every kneading automata group has a faithful representation as a kneading automata group in reduced form.
\end{lem}
\begin{proof}
For any $k$, conditions (1) and (2) of the definition of a kneading automata group are easily verified to hold for $r_k(G)$. To prove (3), we show that $(\pi_k(y))_{y\in Y\setminus\{1\}}$ is a tree-like sequence in $\Sym(X^k)$ for any $k\geq 1$.  We use that a finite connected graph is a tree if and only if it has Euler characteristic $1$. Set $Z:=Y\setminus\{1\}$

Let us begin by making several observations. Let $\Gamma_{i}$ be the cycle graph of $(\pi_{i}(y))_{y\in Z}$.  Define $l_i:Z\rightarrow \Nb$ by setting $l_i(g)$ to be the number of orbits of $\grp{g}$ on $X^i$. For $g\in Z$, let $Y_{g,i}:=\{g_x\mid x\in X^i\}\setminus \{1\}$. As $G$ is a kneading automata group, it follows that $Z=\sqcup_{g\in Z}Y_{g,i}$.

We observe that the functions $l_{i+1}$ are defined in terms of $l_1$ and $l_i$. Fix $g\in Z$. The difference $l_i(g)-|Y_{g,i}|$ is the number of orbits $O$ of $\grp{g}$ on $X^i$ such that $g_x=1$ for every $x\in O$. For each such orbit $O$ and $x\in X$, the set of words $Ox$ is an orbit of $\grp{g}$ on $X^{i+1}$. There are thus $|X|(l_i(g)-|Y_{g,i}|)$ many orbits of $\grp{g}$ on $X^{i+1}$ of this type. For any other orbit $O$, let $h\in Y_{g,i}$ be the non-trivial section of $g$ for some $x\in O$; the orbit $O$ for a given $h$ is unique. For any $W$ an orbit of $\grp{h}$ acting on $X$, the argument in the second paragraph of the proof of Lemma~\ref{lem:kneading_BAG} (which only uses conditions (1) and (2) of a kneading automata group) shows that $OW$ is an orbit of $\grp{g}$ acting on $X^{i+1}$. We conclude that
\[
l_{i+1}(g)=|X|\left(l_i(g)-|Y_{g,i}|\right)+\sum_{h\in Y_{g,i}}l_1(h).
\]

Note further that $\Gamma_i$ is connected for every $i$, since $G$ acts transitively on $X^i$. Additionally,
\[
|V\Gamma_i|=|X|^i+\sum_{g\in Z}l_i(g)\text{ and } |E\Gamma_i|=|Z||X|^i.
\]

We now argue by strong induction on $i\geq 1$ for the claim that $(\pi_i(y))_{y\in Z}$ is a tree-like sequence. The base case holds by definition. Suppose $(\pi_i(y))_{y\in Z}$ is a tree-like sequence for all $i\leq k$. To verify the inductive claim, it suffices to show that $|V\Gamma_{k+1}|-|E\Gamma_{k+1}|=1$.

The inductive hypothesis ensures that
\[
\sum_{g\in Z}l_i(g)=|Z||X|^i-|X|^i+1
\]
for any $i\leq k$. For $k+1$, we have
\[
|V\Gamma_{k+1}|=|X|^{k+1}+\sum_{g\in Z}l_{k+1}(g).
\]
In view of the relationship between the functions $l_i$, we substitute to obtain
\[
\begin{array}{rcl}
|V\Gamma_{k+1}|& =&  |X|^{k+1}+\sum_{g\in Z}\left( |X|\left(l_k(g)-|Y_{g,k}|\right)+\sum_{h\in Y_{g,k}}l_1(h)\right)\\
& = & |X|^{k+1}+|X|\sum_{g\in Z}l_k(g)-|X||Z|+\sum_{g\in Z}l_1(g)\\
& = & |X|^{k+1} +|X|^{k+1}|Z|-|X|^{k+1}+|X|-|X||Z|+|Z||X|-|X|+1\\
& = & |X|^{k+1}|Z|+1.
\end{array}
\]
Hence, $|V\Gamma_{k+1}|-|E\Gamma_{k+1}|=1$, as required. Therefore, $r_k(G)$ is kneading automata group.

Lemma~\ref{lem:kneading_BAG} ensures that $G$ is a self-replicating bounded automata group, and by Corollary~\ref{cor:reduction-reduced-form}, there is $k$ such that $r_k(G)$ is in reduced form. Hence, $r_k(G)$ is a kneading automata group in reduced form.
\end{proof}

\subsection{Elementary amenable kneading automata groups}

We begin by extracting a result from the proof of \cite[Theorem 8.2]{J15}. For $g\in \Aut(\Xt)$, let $\fix_k(g)$ be the collection of fixed points of $\grp{g}$ acting on $X^k$.

\begin{lem}\label{lem:inf_dihedral_char}
Let $G=\grp{Y}\leq \Aut(\Xt)$ be a kneading automata group in reduced form with $|X|>3$. If there are distinct $g$ and $h$ in $Y$ such that $|\fix_1(g)|+|\fix_1(h)|=2$, then either
\begin{enumerate}
\item $|\fix_2(g')|+|\fix_2(h')|>3$ for all distinct, non-trivial $g'$ and $h'$ in $Y$, or
\item $G$ is the infinite dihedral group.
\end{enumerate}
\end{lem}
\begin{proof}
Since $G$ is self-replicating, it is infinite.
Suppose that $(1)$ fails. Say that $g'$ and $h'$ in $Y$ are such that $|\fix_2(g')|+|\fix_2(h')|\leq 3$. Lemma~\ref{lem:kneading_reduced_form} tells us that $r_2(G)$ is a kneading automata group, so Lemma~\ref{lem:fixed_pts_tree-like} implies that $\pi_2(d)$ acts trivially on $X^2$ for all $d\in Y\setminus\{g',h'
\}$.  If $h$ or $g$ is an element of $Y\setminus\{g',h'\}$, then $h$ or $g$ fixes $X$, which contradicts that $|\fix_1(g)|+|\fix_1(h)|=2$. We conclude that $\{g,h\}=\{g',h'\}$. In particular, $|\fix_2(g)|\leq 3$ and $|\fix_2(h)|\leq 3$.

 As $G$ is in reduced form, all finitary elements of $Y$ have depth one, so there are no non-trivial finitary elements in $Y\setminus\{g,h\}$, since all elements of $Y\setminus\{g,h\}$ act trivially on $X$. If $d\in Y\setminus\{g,h\}$ is directed, then the previous paragraph ensures that it fixes $X^2$. Every section $d_x$ for $x\in X\setminus\{v^d\}$ is trivial  where $v^d$ is the active vertex for $d$ on level one, as $d$ fixes $X^2$. The element $d$ therefore acts trivially on $\Xt$, and $d=1$. Hence, $Y=\{g,h,1\}$.

Take $\Gamma$ the cycle graph for $(\pi_1(g),\pi_1(h))$ and set $n:=|X|$. Let $K_g$ and $K_h$ be the number of non-trivial cycles in $g$ and $h$, respectively, and $F_g$ and $F_h$ be the number of trivial cycles in $g$ and $h$, respective. Euler's formula for trees implies that $|V\Gamma|-|E\Gamma|=1$. In view of the definition of $\Gamma$, we see that
\[
|V\Gamma|=n+K_g+F_g+K_h+F_h
\]
and $|E\Gamma|=2n$. Therefore, $K_g+K_h=n-1$. Observe that $|\fix_1(g)\cup \fix_1(h)|=2$, as $G=\grp{g,h}$ acts transitively on $X$ and hence $g,h$ do not have a common fixed point.  If $\fix_1(g)\cup \fix_1(h)=\{x,x'\}$, then it follows that every element of $X\setminus \{x,x'\}$ belongs to both a non-trivial cycle of $\pi_1(g)$ and a non-trivial cycle of $\pi_1(h)$ and $x,x'$ each belong to a non-trivial cycle of exactly one of $\pi_1(g)$ or $\pi_1(h)$.  It follows that $2n-2\geq 2K_g+2K_h$ with equality if and only if $\pi_1(g),\pi_1(h)$ are both products of disjoint two-cycles in their respective cycle decompositions (omitting trivial cycles).  As $K_g+K_h=n-1$, we deduce that $\pi_1(g)$ and $\pi_1(h)$ are products of disjoint two-cycles, so $|\pi_1(g)|=|\pi_1(h)|=2$.

We now have two cases: (a) $Y$ contains a non-trivial finitary element, and (b) $Y$ contains no non-trivial finitary element. For case (a), suppose that $g$ is finitary. The element $g$ therefore has order two, since $G$ is in reduced form. If $g$ fixes a point $u\in X$, then $g$ must fix pointwise $uX$ in $X^2$, since $g_u=1$. This implies that $|\fix_2(g)|>3$, which is absurd. We conclude that $g$ acts fixed point freely on $X$, so $h$ fixes two points $v_1,v_2$ of $X$.  If some $h_{v_i}=1$, then $h$ fixes $v_iX$ contradicting that  $|\fix_2(h)|\leq 3$.  Since $G$ is a kneading automata group with self-similar generating set $Y=\{g,h,1\}$, we deduce that, up to relabelling, $h_{v_1}=h$, $h_{v_2}=g$ and all other sections of $h$ are trivial. Since $\pi_1(h)$ is a product of disjoint two-cycles and $g$ is an involution, it now follows that $h^2=1$. Hence, $G$ is an infinite dihedral group.

For case (b), both $g$ and $h$ are directed, and as $Y$ is self-similar, neither $g$ nor $h$ has non-trivial finitary sections. Take $u\in X\setminus\{v^g\}$. If $g$ fixes $u$, then $|\fix_2(g)|>3$, since $g_u=1$; this is absurd. The only possible fixed point of $g$ in $X$ is its active vertex $v^g$. Likewise, the only possible fixed point of $h$ in $X$ is its active vertex $v^h$. Therefore, $g$ fixes $v^g$, and $h$ fixes $v^h$. We conclude that $g$ and $h$ are torsion with $|g|=|h|=2$. The group $G$ is thus infinite and generated by two involutions, so $G$ is an infinite dihedral group.
\end{proof}

We are now ready to prove the desired theorem.

\begin{thm}\label{thm:kneading}
Suppose that $G=\grp{Y}\leq \Aut(\Xt)$ is a kneading automata group. If $G$ is elementary amenable, then $G$ is virtually abelian.
\end{thm}
\begin{proof}
Suppose toward a contradiction that $G$ is not virtually abelian. We begin with several reductions. In view of Lemma~\ref{lem:kneading_BAG}, $G$ is a self-replicating bounded automata group, and by Lemma~\ref{lem:kneading_reduced_form}, we may assume that $G$ is in reduced form. By passing to $r_2(G)$, where $r_2:\Aut(\Xt)\rightarrow \Aut((X^2)^*)$ is the canonical inclusion, we may assume that $|X|>3$. If $g,h\in Y$ are distinct elements such that $|\fix_1(g)|+|\fix_1(h)|=2$, then Lemma~\ref{lem:inf_dihedral_char} ensures that $|\fix_2(g')|+|\fix_2(h')|>3$ for all distinct, non-trivial $g'$ and $h'$ in $Y$. Passing to $r_2(G)$ a second time, we may additionally assume that $|\fix_1(g)|+|\fix_1(h)|>3$ for all distinct, non-trivial $g$ and $h$ in $Y$.

Applying Corollary~\ref{cor:ext_EA_3}, we obtain weakly self-replicating $M\normal G$ and $F\normal G$ such that $\rk(M)+2=\rk(G)=\rk(F)+1$, $M\leq F$, $F/M$ is abelian, and $G/F$ is finite.  By Lemma~\ref{lem:weakly_self-similar_trans}, $F$ also acts transitively on $X$. We now have two cases (1) every $g\in Y$ fixes at least two points on $X$, and (2) some unique $g\in Y$ fixes one or fewer points on $X$.

Let us suppose first that every $g\in Y$ has at least two fixed points. Fix $h\in Y$ directed. Since $h$ fixes at least two points, we may find $x\in X$ such that $h$ fixes $x$ and $h_x=:\sigma$ is finitary. The section $\sigma$ is again an element of $Y$, so it fixes some $x'\in X$, as it also fixes two points. The element $h$ therefore fixes $xx'$, and $h_{xx'}=\sigma_{x'}=1$, since $G$ is in reduced form. Let $v^h$ be the active vertex of $h$ in $X^2$ and $d \geq 1$ be least such that $h^d$ fixes $v^h$. For $O$ the orbit of $v^h$ under $\grp{h}$, that $G$ is a kneading automata group ensures that $v^h$ is the only element of $O$ for which $h$ has a non-trivial section.  It follows that $(h^d)_{v^h}=h$. Since $(h^d)$ also fixes $xx'$, Lemma~\ref{lem:com_contain_sections_2} implies that $h\in F$.

For any non-trivial section $h_z$ of $h$, the definition of the kneading automata ensures that $h_z$ is the only non-trivial section of $h$ on the orbit for $z$ under $\grp{h}$. Taking $k\geq 1$ least such that $h^k$ fixes $z$, it follows that $(h^k)_z=h_z$. Recalling that $F$ is weakly self-similar, we conclude that $h_z\in F$. It now follows from Proposition~\ref{prop:self-rep_char_gen} (as $G$ is self-replicating) that $F=G$, which contradicts the rank of $G$.

Suppose next that there is a unique $g\in Y$ that has at most one fixed point on $X$. Fix $h\in Y\setminus \{g\}$ directed. The element $h$ has at least three fixed points in $X$, so there is $x\in X$ such that $h$ fixes $x$, $h_x$ is finitary, and $h_x\neq g$. The section $h_x$ fixes some $x'\in X$, as it differs from $g$. Hence, $h$ fixes $xx'$, and $h_{xx'}=1$. Just as in the previous case, it now follows that $h\in F$. The quotient $F/M$ is abelian, and $F$ acts on $X$ transitively. As Lemma~\ref{lem:com_contain_sections_2} also applies in this setting, we can run the argument again to deduce that $h\in M$, and as in the previous case,  $h_x\in M$ for all $x\in X$.  We conclude that $h_x\in M$ for all $x\in X$ and $h\in Y\setminus \{g\}$. The subgroup $M$ therefore contains $Y\setminus(\{g_x\mid x\in X\}\cup \{g\})$.

Let us now consider the element $g$. The quotient $G/M$ cannot be abelian because $\rk(G)=\rk(M)+2$ and so $\{g_x\mid x\in X\}\cup \{g\}$ is not equal to $\{1,g\}$. The element $g$ thereby admits non-trivial finitary sections and, in particular, is not, itself, finitary. Fix $x\in X$ such that $g_x$ is a non-trivial finitary element $\sigma$ and note that $\sigma\in Y\setminus \{g\}$ fixes a point $x'\in X$.

Let $d\geq 1$ be least such that $g^d$ fixes $v:=v^g$ and let $m\geq 1$ be least such that $g^m$ fixes $x$. Since $G$ is a kneading automaton group, and hence each orbit of $g$ has at most one non-trivial section, $(g^d)_v=g$ and $(g^m)_x=\sigma$.  Observe that $g^{d^2}$ fixes $vv$, $g^m$ fixes $xx'$
$(g^m)_{xx'}=1$ and $(g^{d^2})_{vv}=g$.   The element $g^{d^2m}$ thus fixes $vv$ and $xx'$. Moreover, $(g^{d^2m})_{xx'}=1$ and $(g^{d^2m})_{vv}=g^m$. Applying Lemma~\ref{lem:com_contain_sections_2},  we deduce that $g^m\in F$. The element $g^{d^2m}$ is therefore also in $F$, and since $F/M$ is abelian, a second application of Lemma~\ref{lem:com_contain_sections_2} ensures that $g^m\in M$. The element $g^m$ fixes $x$, and $(g^m)_{x}=\sigma$. Therefore, $\sigma\in M$, since $M$ is weakly self-similar. We conclude that $M$ contains every finitary section of $g$ and hence $Y\setminus \{g\}$. The quotient $G/M$ is thus abelian, which contradicts the rank of $G$.
\end{proof}

Theorems~\ref{thm:iterated_monodromy} and \ref{thm:kneading} yield an immediate corollary.

\begin{cor}\label{cor:monodromy}
Every iterated monodromy group of a post-critically finite polynomial is either virtually abelian or not elementary amenable.
\end{cor}

\section{Generalized basilica groups}

\begin{defn}
We say that $G\leq\Aut(\Xt)$ is a \textbf{generalized basilica group} if $G$ admits a finite generating set $Y$ such that for every $g\in Y$ either $g_x=1$ for all $x\in X$ or $g_x\in \{1,g\}$  for all $x\in X$ with exactly one $x$ such that $g_x=g$. We call $Y$ a \textbf{distinguished generating set} and write $G=\grp{Y}\leq\Aut(\Xt)$ to indicate that $Y$ is a distinguished generating set for $G$.
\end{defn}

A straightforward verification shows that a generalized basilica group is a bounded automata group.

\subsection{A reduction theorem}
The reduction result established herein reduces many questions for generalized basilica groups to the self-replicating case. It follows by induction on the number of directed elements of a distinguished generating set.

Our first lemma gives a tool to reduce to groups with possibly fewer number of directed generators, based on the orbits of the group on the first level of the tree.

\begin{lem}\label{lem:reduction_GBgroups}
	Suppose that $G=\grp{Y}\leq \Aut(\Xt)$ is a generalized basilica group, let $O_1,\dots,O_n$ list the orbits of $G$ on $X$, and let $Y_i\subseteq Y$ be the collection of $g\in Y$ such that $v^g\in O_i$. For each $i$, $\grp{Y_i}\leq \Aut(\Xt)$ is a generalized basilica group generated by directed elements, and
	\[
	G\hookrightarrow \prod_{i=1}^n\Sym(O_i)\ltimes \prod_{i=1}^n\grp{Y_i}^{O_i} .
	\]
\end{lem}
\begin{proof}
That each $\grp{Y_i}\leq \Aut(\Xt)$ is a generalized basilica group generated by directed elements is immediate.

For each $x\in O_i$ and $h\in Y$, the section $h_x$ is an element of $\grp{Y_i}$. Inducting on the word length, one sees that $g_x\in \grp{Y_i}$ for all $g\in G$ and $x\in O_i$. Letting $\pi_1:G\rightarrow \Sym(X)$ be the induced homomorphism, we see that $\pi_1(G)\leq \prod_{1=1}^n\Sym(O_i)$. We may now define
\[
	G\rightarrow \prod_{i=1}^n\Sym(O_i)\ltimes \prod_{i=1}^n\grp{Y_i}^{O_i} .
\]
by $g\mapsto \pi_1(g)((g_x)_{x\in O_i})_{i=1}^n$. One verifies that this function is a monomorphism.
\end{proof}

Our next lemmas address the case in which Lemma~\ref{lem:reduction_GBgroups} does not reduce the number of directed elements.

\begin{lem}\label{lem:almost_self-replicating_0}
	Suppose that $G=\grp{Y}\leq\Aut(\Xt)$ is a generalized basilica group. If $O\subseteq X$ is setwise preserved by the action of $G$ on $X$, then $O^*$ is setwise preserved by $G$, and the induced homomorphism $\phi:G\rightarrow \Aut(O^*)$ is such that $\phi(G)$ is a generalized basilica group with distinguished generating set $\phi(Y)$.
\end{lem}
\begin{proof}
That $G$ is self-similar ensures that $O^*$ is setwise preserved by $G$. Take $g\in G$ and $w\in O^*$. By definition, $g_w=\psi_{g(w)}^{-1}\circ g\circ \psi_w$ where $\psi_v:X^*\rightarrow vX^*$ by $x\mapsto vx$. The map $\phi$ is nothing but restriction to $O^*$, so we see that
	\[
	\phi(g_w)=(g_w)\rest_{O^*}=\psi_{g(w)}^{-1}\rest_{O^*}\circ g\rest_{O^*}\circ \psi_w\rest_{O^*}=\phi(g)_w.
	\]
It now follows that $\phi(G)\leq \Aut(O^*)$ is a generalized basilica group with generating set $\phi(Y)$.
\end{proof}

We call the action of $G$ on $O^*$ a \textbf{sub-basilica action} of $G$ induced from $\Xt$.

\begin{lem}\label{lem:loc.finite.sym}
Let $O\subseteq X$ and let $G$ be the subgroup of $\Aut(\Xt)$ consisting of those elements that fix $O^*$ pointwise and have trivial sections outside of $O^*$.  Then $G\cong \sym(X\setminus O)^{O^*}$ and hence is locally finite.
\end{lem}
\begin{proof}
It is clear that $G$ is a subgroup of $\Aut(\Xt)$  since if $f,g\in G$ and $w\notin O^*$, we have $(fg)_w= f_{g(w)}g_w=1$ as $g(w)\notin O^*$.  There is a well-known set-theoretic bijection of $\Aut(\Xt)$ with $\sym(X)^{X^*}$ sending an element $g$  to its `portrait' $(\pi_1(g_w))_{w\in X^*}\in \sym(X)^{X^*}$. Put $Y:=X\setminus O$. An element $g\in \Aut(\Xt)$ belongs to $G$ if and only if $g_w=1$ for $w\notin O^*$ and $\pi_1(g_w)\in \sym(Y)$ for $w\in O^*$ (where $\sym(Y)$ is viewed as a subgroup of $\sym(X)$ in the usual way).  There is thus a bijection $\pi\colon G\to \sym(Y)^{O^*}$ given by $\pi(g)=(\pi_1(g_w))_{w\in O^*}$.  The mapping $\pi$ is a homomorphism because if $f,g\in G$ and $w\in O^*$, then $(fg)_w=f_{g(w)}g_w=f_wg_w$ as $g$ fixes $O^*$ pointwise. This proves the first statement.  Since the variety of groups generated by any finite group is locally finite, it follows that $\sym(Y)^{O^*}$ is locally finite.
\end{proof}

\begin{lem}\label{lem:almost_self-replicating}
	Suppose that $G=\grp{Y}\leq\Aut(\Xt)$ is a generalized basilica group generated by directed elements. If there is an orbit $O\subseteq X$ of $G$ such that $v^g\in O$ for all $g\in Y$, then the homomorphism $\phi:G\rightarrow \Aut(O^*)$ induced by the action $G\acts O^*$ enjoys the following properties:
	\begin{enumerate}
	\item $\phi(G)$ is a self-replicating generalized basilica group with generating set $\phi(Y)$, and
	\item $\ker(\phi)$ is locally finite.
	\end{enumerate}
\end{lem}
\begin{proof}
	For (1), Lemma~\ref{lem:almost_self-replicating_0} ensures that $\phi(G)\leq \Aut(O^*)$ is a generalized basilica group with generating set $\phi(Y)$. Lemma~\ref{lem:section_sgrp} implies that $\sec_G(v)=G$ for all $v\in O$, so $\sec_{\phi(G)}(v)=\phi(G)$ for all $v\in O$.  Since $G$ acts transitively on $O$, $\phi(G)$ is self-replicating.

For (2), clearly $\ker(\phi)$ fixes $O^*$ pointwise.  Since $v^g\in O$ for all $g\in Y$, it follows that $f_v=1$ for all $v\in X\setminus O$ and $f\in G$.
As $G$ is self-similar, it follows that $f_w=1$ for all $f\in G$ and $w\notin O^*$.  Indeed, we may write $w=uav$ with $u\in O^*$, $a\in X\setminus O$, and $v\in X^*$, so $f_w = ((f_u)_a)_v = 1_v=1$.  Thus, $\ker(\phi)$ is isomorphic to a subgroup of $\sym(X\setminus O)^{O^*}$, and hence is locally finite, by Lemma~\ref{lem:loc.finite.sym}. 	
\end{proof}

Bringing together our lemmas produces the desired reduction theorem. To state our main theorem in full generality requires a definition.
\begin{defn}
Suppose that $Q$ is a property of generalized basilica groups. We say that $Q$ is a \textbf{sub-basilica stable} property if for every generalized basilica group $G=\grp{Y}\leq\Aut(\Xt)$ with property $Q$, the following holds: For every generalized basilica group of the form $\grp{Y'}\leq \Aut(\Xt)$ with $Y'\subset Y$, any generalized basilica group $\widetilde{\grp{Y'}}$ induced by a sub-basilica action of $\grp{Y'}$ has property $Q$. In particular, $\grp{Y'}$ has $Q$.
\end{defn}
Observe that any property of groups stable under passing to subgroups and quotients is sub-basilica stable.

\begin{thm}\label{thm:reduction_thm_GB}
Let $Q$ be a sub-basilica stable property  and suppose that $P$ is a property of groups enjoyed by abelian groups and stable under taking subgroups, finite direct products, and $P$-by-finite groups. If every self-replicating generalized basilica group with property $Q$ has property $P$, then every generalized basilica group with property $Q$ is (locally finite)-by-$P$.
\end{thm}
\begin{proof}
Suppose that $G=\grp{Y}\leq\Aut(\Xt)$ is a generalized basilica group with property $Q$. We argue by induction on the number of directed elements in $Y$, denoted by $\delta(Y)$, for the theorem. If $\delta(Y)\leq 1$, then $G$ is virtually abelian by Lemma~\ref{lem:reduce_to_directed}, and we are done. Suppose the theorem holds up to $n$ and say that $\delta(Y)=n+1$. Let $O_1,\dots,O_l$ list the orbits of $G$ on $X$ and let $Y_i\subseteq Y$ be the collection of directed $d\in Y$ such that $v^d\in O_i$. Lemma~\ref{lem:reduction_GBgroups} supplies an injection
 		\[
 		G\hookrightarrow \prod_{i=1}^l\Sym(O_i)\ltimes \prod_{i=1}^lH_i^{O_i}.
 		\]
It is enough to show that each $H_i$ is (locally finite)-by-$P$
 		
 	Each $H_i\leq \Aut(\Xt)$ is again a generalized basilica group with property $Q$, since $Q$ is sub-basilica stable. If no $O_i$ contains $v^d$ for every directed $d\in Y$, each $H_i$ is such that $\delta(Y_i)<n+1$. The inductive hypothesis then ensures that each $H_i$ is (locally finite)-by-$P$.
 	
 	Suppose some $O_i$ contains $v^d$ for every directed $d\in Y$.  Without loss of generality, $i=1$, so $H_j=\{1\}$ for $j>1$. Clearly each $H_j$ with $j>1$ is (locally finite)-by-$P$. We apply Lemma~\ref{lem:almost_self-replicating} to the group $H_1$. Letting $K:=\ker(H_1\acts O_1^*)$, Lemma~\ref{lem:almost_self-replicating} ensures that $H_1/K\leq \Aut(O_1^*)$ is a self-replicating generalized basilica group, and $K$ is locally finite. Since $H_1/K$ is the group induced by a sub-basilica action, $H_1/K$ has property $Q$, hence $H_1/K$ has property $P$. We conclude that $H_1$ is (locally finite)-by-$P$. 	
\end{proof}

We can do better for torsion-free bounded automata groups. The proof is essentially the same as that of Theorem~\ref{thm:reduction_thm_GB}, so we leave it to the reader.
\begin{thm}
Suppose that $P$ is a property of groups enjoyed by abelian groups and stable under taking subgroups, finite direct products, and $P$-by-finite groups. If every self-replicating torsion-free bounded automata group has property $P$, then every torsion-free bounded automata group has $P$.
\end{thm}

\subsection{Groups with center}
A \textbf{system of imprimitivity} for a group action $G\acts X$ is a $G$-equivariant equivalence relation on $X$. That is, $x\sim y$ if and only if $g(x)\sim g(y)$. The equivalence classes of a system of imprimitivity are called \textbf{blocks of imprimitivity}. A block $B$ is called \textbf{trivial} if it equals either $X$ or a singleton set.

\begin{lem}\label{lem:sections_blocks}
Suppose that $h\in \Aut(\Xt)$ is a directed automorphism such that $h_x\in \{1,h\}$ for all $x\in X$. Let $B\subseteq X$ be a non-trivial block of imprimitivity for the action of $\grp{h}$ on $X$ and $O$ be the orbit of $v^h$ under the action of $\grp{h}$ on $X$. If $|O\cap B|\geq 2$, then there is $j\in \Zb$ and $w\in B$ such that $h^j$ stabilizes $B$ setwise, $(h^j)_w=h^{\pm 1}$, and $(h^{-j})_w=1$.
\end{lem}
\begin{proof}Let $i\geq 0$ be least such that $h^i(v^h)\in B$.

Suppose first that $i=0$, so $v^h\in B$. Set $w:=v^h$.  Let $j>0$ be least such that $h^j(w)\in B$ and observe that $h^j(w)\neq w$ since $|O\cap B|\geq 2$.  The element $h^j$ stabilizes $B$ setwise, and it is clear that $(h^j)_{w}=h$. On the other hand, $(h^{-j})$ has non-trivial sections at $\{h(w),\dots, h^j(w)\}$, and this set excludes $w$. We conclude that $(h^{-j})_{w}=1$, establishing the lemma in this case.

Suppose next that $i>0$ and set $w:=h^i(v^h)$. Let $j>0$ be least such that $h^{j+i}(v^h)\in B$. The element $h^{-j}$  setwise stabilizes the block $B$, so  $h^{-j}(w)=h^{i-j}(v^h)\in B$. If $j\leq i$, then we contradict the choice of $i$, so $j>i$. It now follows that $(h^{-j})_{w}=h^{-1}$. On the other hand, $h^j(w)=h^{j+i}(v^h)$. As $j+i>j$, $h^{-j}$ has a trivial section at $h^j(w)$. We deduce that $(h^j)_w=1$.
\end{proof}

\begin{lem}\label{lem:commute}
	Let $G\leq \Aut(X^*)$ be self-similar, $1\neq d\in G$ be directed such that $d_x\in \{1,d\}$ for all $x\in X$, and $O$ be the orbit of $v^d\in X$ under the action of $\grp{d}$. Then, $O$ is setwise stabilized by $C_G(d)$, and $g_v\in C_G(d)$ for all $g\in C_G(d)$ and $v\in O$ .
\end{lem}
\begin{proof}
	Take  $n\geq 1$ least such that $d^n$ fixes $O$. The element $d^n$ is such that
	\[
	(d^n)_v=
	\begin{cases}
	d & \text{ if } v\in O\\
	1 & \text{ else. }
	\end{cases}
	\]
	For any $v\in O$ and $z\in C_G(d)$,  we see that
	\[
	(d^{-n})_{z(v)}z_vd=(d^{-n}zd^{n})_v=z_v.
	\]
	The sections of $d^{-n}$ are either $d^{-1}$ or $1$, and the non-trivial sections occur for exactly the $v\in O$. It is therefore the case that $z(v)\in O$ and $d$ commutes with $z_v$. The group $C_G(d)$ thus setwise stabilizes $O$, and $g_v\in C_G(d)$ for all $g\in C_G(d)$ and $v\in O$.
\end{proof}

\begin{lem}\label{lem:wself-similar}
	For $G\leq \Aut(X^*)$, if $v\in \Xt$ is such that $\sec_G(v)=G$, then $\sec_{Z(G)}(v)\leq Z(G)$. In particular, if $G$ is weakly self-replicating, then $Z(G)$ is weakly self-similar.
\end{lem}
\begin{proof}
	We see that $Z(G)_{(v)}\leq Z(G_{(v)})\leq G_{(v)}$. Therefore,
	\[
	\sec_{Z(G)}(v)=\phi_v(Z(G)_{(v)})\leq\phi_v(Z(G_{(v)}))\leq Z(\phi_v(G_{(v)}))=Z(G).
	\]
The lemma now follows.
\end{proof}

\begin{thm}\label{thm:center_self-replicating}
	Suppose that $G=\grp{Y}\leq \Aut(X^*)$ is a self-replicating generalized basilica group. If $Z(G)$ is non-trivial, then $G$ is abelian.
\end{thm}
\begin{proof}
By Proposition~\ref{prop:self-rep_char}, we may assume that $Y$ consists of directed elements. Set $Z:=Z(G)$ and for $O\subseteq X$, define
\[
S_O:=\{z_v\mid z\in Z\text{ and }v\in O\}\cup Z.
\]

 Fix $d\in Y$ and let $O_d$ be the orbit of $\grp{d}$ on $X$ such that $v^d\in O_d$.  Lemma~\ref{lem:commute} ensures that $C_G(d)$ contains $S_{O_d}$, so $d\in C_G(S_{O_d})$. Let $W\subseteq X$ be maximal such that $O_d\subseteq W$ and $d\in C_G(S_W)$. We argue that $W=X$.

Let us first see that $W$ is invariant under the action of $Z$. Take $v\in W$ and $z\in Z$. For $g\in Z$, we see that $ (gz)_v=g_{z(v)}z_v$. Since $d$ commutes with $(gz)_v$ and $z_v$, it follows that $d$ commutes with $g_{z(v)}$. The element $d$ thus commutes with $g_{z(v)}$ for every $g\in Z$, hence $z(v)\in W$, since $W$ is maximal.

We next argue that $W$ is invariant under each $c\in Y$. Fixing $c\in Y$, we have two cases. For the first case, suppose that $O_c\cap W=\emptyset$, where $O_c\subseteq X$ is the orbit of $v^c$ under $\grp{c}$. Take $g\in Z$ and fix $v\in W$. We see that
\[
g_v=(c^{-1}gc)_v=(c^{-1})_{gc(v)}g_{c(v)}c_v=(c^{-1})_{cg(v)}g_{c(v)}
\]
Since $W$ is invariant under $Z$, $g(v)\in W$. The element $cg(v)$ is thus not in $O_c$, so $(c^{-1})_{cg(v)}=1$. We conclude that $g_v=g_{c(v)}$, hence $d$ commutes with $g_{c(v)}$. Therefore, $c(v)\in W$ by maximality, and $W$ is invariant under the action of $c$.

For the second case, suppose that $O_c\cap W\neq \emptyset$. Applying Lemma~\ref{lem:commute}, $O_c$ is $Z$-invariant. The intersection $O_c\cap W$ is then also $Z$-invariant, so $|O_c\cap W|\geq 2$, since $Z$ cannot fix any vertex of $X$ by Lemmas~\ref{lem:fix} and \ref{lem:wself-similar}. The orbits of $Z$ on $O_c$ form a non-trivial system of imprimitivity for $\grp{c}$, so there is a non-trivial block of imprimitivity $B\subseteq O_c\cap W$ for the action of $\grp{c}$ on $X$. Via Lemma~\ref{lem:sections_blocks}, we may find $w\in B$ and $j\in \Zb$ such that $(c^j)_w=c^{\pm 1}$, $(c^{-j})_w=1$, and $c^j$ setwise fixes $B$.

Let $z\in Z$ be such that $z(c^j(w))=w$. The element $c^{-j}zc^j$ is an element of $Z$, and $w\in W$. The element $d$ thus commutes with $(c^{-j}zc^j)_w$.  On the other hand,
\[
(c^{-j}zc^{j})_w=(c^{-j})_{w}z_{c^j(w)}(c^j)_{w}=z_{c^j(w)}c^{\pm 1},
\]
so $d$ commutes with $z_{c^j(w)}c^{\pm 1}$. Furthermore, $d$ commutes with $z_{c^j(w)}$ since $c^j(w)\in B\subseteq  W$. We deduce that $d$ commutes with $c$.

For any vertex $u\in W$ and $g\in Z$, $d$ commutes with $(c^{-1}gc)_u$. On the other hand,
\[
(c^{-1}gc)_u=(c^{-1})_{g(c(u))}g_{c(u)}c_{u}
\]
Since the sections of $c$ are either equal $c$ or trivial, $d$ commutes with $ c_{u}$ and $(c^{-1})_{g(c(u))}$. We conclude that $d$ commutes with $g_{c(u)}$, so $c(u)\in W$  by the choice of $W$. The set $W$ is thus setwise fixed by every $c\in Y$. Our group $G=\grp{Y}$ acts transitively on $X$, so it is the case that $W=X$.

For any $c\in Y$, there is $z\in Z$ such that $z(c(v^c))\neq c(v^c)$ by Lemma~\ref{lem:fix}. Taking such a $z$, $c^{-1}zc\in Z$, and we have that $d$ commutes with $(c^{-1}zc)_{v^c}$, by our work above. Furthermore,
\[
(c^{-1}zc)_{v^c}=z_{c(v^c)}c,
\]
and $d$ commutes with $z_{c(v^c)}$. We infer that $d$ commutes with $c$.

Since $d$ is arbitrary, it now follows that all elements of $Y$ commute, hence $G$ is abelian.
\end{proof}

We can  now upgrade Theorem~\ref{thm:trans} in the self-replicating case. We note that Theorem~\ref{thm:trans} cannot be upgraded itself since one can always add finitary elements.
\begin{cor}\label{cor:trans_virt-ab}
Suppose that $G=\grp{Y}\leq \Aut(\Xt)$ is a self-replicating generalized basilica group. If $G$ contains an odometer, then either $G$ is abelian or $G$ is not elementary amenable.
\end{cor}
\begin{proof}
Suppose that $G$ is elementary amenable. By Theorem~\ref{thm:trans}, $G$ is virtually abelian. Suppose for contradiction that $G$ is not abelian. Applying Corollary~\ref{cor:ext_EA}, there is $H\normal G$ such that $H$ is weakly self-replicating, $\rk(H)+1=\rk(G)$, and $G/H$ is finite. Observe that $H$ is abelian, since it must be infinite with rank $0$.

Take $h\in G$ an odometer. Since $G/H$ is finite, there is $i>0$ such that $h^i\in H$. Set $m:=|X|i$. The element $h^m$ is in $H$, and $(h^m)_x=h^i$ for all $x\in X$. Since $H$ is abelian, $f^{-1}h^mf=h^m$ for any $f\in H$. For any $x\in X$, it is then the case that $(f_x)^{-1}h^if_x=h^i$. We conclude that $C_G(h^i)$ contains $H$ along with $f_x$ for all $f\in H$ and $x\in X$.

Take $d\in Y$.  In view of Lemma~\ref{lem:weakly_self-similar_trans}, there is $f\in H$ such that $f(d(v^d)))\neq d(v^d)$. The element $d^{-1}fd$ is in $H$, and $(d^{-1}fd)_{v^d}=f_{d(v^d)}d$. The centralizer $C_G(h^i)$ therefore contains $f_{d(v^d)}d$. Since $C_G(h^i)$ also contains $f_{d(v^d)}$, we infer that $d\in C_G(h^i)$. The element $h^i$ is therefore central in $G$.

The group $G$ thus has non-trivial center. Applying Theorem~\ref{thm:center_self-replicating}, we conclude that $G$ is abelian, which is absurd.
\end{proof}

Theorem~\ref{thm:center_self-replicating} and Theorem~\ref{thm:reduction_thm_GB} together imply all nilpotent generalized basilica groups are virtually abelian.

 \begin{cor}\label{cor:nilpotent}
 Every nilpotent generalized basilica group is virtually abelian.
 \end{cor}
 	\begin{proof}
 	Nilpotent groups have non-trivial centers. Theorem~\ref{thm:center_self-replicating} thus implies that every self-replicating nilpotent generalized basilica group is virtually abelian.
 	
 		The property $Q$ of being nilpotent is closed under subgroups and quotients, so a fortiori it is a sub-basilica stable property. The property $P$ of being virtually abelian is enjoyed by abelian groups and stable under taking subgroups, finite direct products, and forming $P$-by-finite groups. We may thus apply Theorem~\ref{thm:reduction_thm_GB} to conclude that every nilpotent generalized basilica group is (locally finite)-by-virtually abelian.
 		
 	Every subgroup of a finitely generated nilpotent group is finitly generated. Hence, every nilpotent generalized basilica group is finite-by-virtually abelian. Since generalized basilica groups are residually finite, it follows that every nilpotent generalized basilica group is virtually abelian.
 	\end{proof}
 	
 	\begin{cor}
 	Every torsion-free nilpotent bounded automata group is virtually abelian.
 	\end{cor}

\subsection{Balanced groups}

\begin{defn}
Let $G=\grp{Y}\leq \Aut(\Xt)$ be a generalized basilica group. We say that $G$ is \textbf{balanced} if for each directed $g\in Y$, the least $n\geq 1$ for which $g^n$ fixes $v^g\in X$ is also such that $g^n$ fixes $X$ pointwise.
\end{defn}
For $G=\grp{Y}\leq \Aut(\Xt)$ balanced, take $g\in Y$ directed and let $c_1\dots c_n$ be the cycle deomposition of $\pi_1(g)\in \Sym(X)$. That $G$ is balanced ensures that for $c_i$ such that $v^g$ appears in $c_i$, the order $|c_j|$ divides $|c_i|$ for all $ 1\leq j\leq n$. This condition gives us control over the generators analogous to how the tree-like condition did for kneading automata groups.

\begin{thm}\label{thm:balanced_self-rep}
Suppose that $G=\grp{Y}\leq \Aut(\Xt)$ is a self-replicating generalized basilica group. If $G$ is balanced, then either $G$ is abelian or $G$ is not elementary amenable.
\end{thm}
\begin{proof} By Proposition~\ref{prop:self-rep_char}, we may assume that $Y$ consists of directed elements.

Suppose that $G$ is elementary amenable. If $Y$ contains an element that acts transitively, i.e., an odometer, then Corollary~\ref{cor:trans_virt-ab} implies that $G$ is abelian, and we are done.

Let us suppose that each element of $Y$ does not act transitively and suppose toward a contradiction that $G$ is not abelian. Applying Corollary~\ref{cor:ext_EA}, there is $H\normal G$ such that $H$ is weakly self-replicating, $\rk(H)+1=\rk(G)$, and $G/H$ is either finite or abelian.

For each $d\in Y$, let $n$ be least such that $d^n$ fixes $v^d$. Since $G$ is balanced, $d^n$ fixes $X$, and as $\grp{d}$ does not at transitively on $X$, there is $x\in Y$ such that $(d^n)_x=1$. In view of Lemma~\ref{lem:com_contain_sections_2}, we deduce that $d\in H$. As $d$ is arbitrary, we conclude that $H=G$ which is absurd since $\rk(H)<\rk(G)$.  The group $G$ is thus abelian.
\end{proof}

\begin{cor}\label{cor:balanced}
Every balanced generalized basilica group is either (locally finite)-by-(virtually abelian) or not elementary amenable.
\end{cor}
\begin{proof}
Theorem~\ref{thm:balanced_self-rep} ensures that every self-replicating  generalized basilica group that is balanced and elementary amenable is virtually abelian.
 	
 		The property $Q$ of being balanced and elementary amenable is a sub-basilica stable property.  The property $P$ of being virtually abelian is enjoyed by abelian groups and stable under taking subgroups, finite direct products, and forming $P$-by-finite groups. We may thus apply Theorem~\ref{thm:reduction_thm_GB} to conclude that every balanced and elementary amenable generalized basilica group is (locally finite)-by-virtually abelian.
\end{proof}
\begin{rmk}
Example~\ref{ssec:balanced} below shows that Theorem~\ref{cor:balanced} is sharp.
\end{rmk}

We conclude this section with a sufficient condition to be a balanced generalized basilica group.
\begin{prop}\label{prop:balanced}
Suppose that $G=\grp{Y}\leq \Aut(\Xt)$ is a generalized basilica group that acts transitively on $X$. If $\pi_1(G)\leq \Sym(X)$ is abelian, then $G$ is balanced.
\end{prop}
\begin{proof}
Let $d\in Y$ be directed. The image $\pi_1(\grp{d})$ is a normal subgroup of $\pi_1(G)$, so the orbits of $\grp{\pi_1(d)}$ form a system of imprimitivity for the action of $\pi_1(G)$ on $X$. Since $\pi_1(G)$ is transitive, all orbits of $\grp{\pi_1(d)}$ have the same size, and it follows that $d$ is balanced.
\end{proof}

\section{Groups of abelian wreath type}

\begin{defn}
We say that $G\leq\Aut(\Xt)$ is of \textbf{abelian wreath type} if $\pi_1(G)\leq \Sym(X)$ is abelian and $G$ admits a finite self-similar generating set $Y$ such that for every $g\in Y$ either $g_x=1$ for all $x\in X$ or $g_x\in \{g\}\cup \pi_1(G)$ for all $x\in X$ with exactly one $x$ such that $g_x=g$. We call $Y$ a \textbf{distinguished generating set} and write $G=\grp{Y}\leq\Aut(\Xt)$ to indicate that $Y$ is a distinguished generating set for $G$.
\end{defn}
A straightforward verification shows that a group of abelian wreath type is a bounded automata group.  Let us note several immediate consequences of abelian wreath type; the proofs are elementary and so are left to the reader.
\begin{obs} Let $G\leq \Aut(\Xt)$ be a self-replicating group of abelian wreath type.
\begin{enumerate}
\item For all non-trivial $f\in \Sym(X)\cap G$ and $v\in X$, $f(v)\neq v$.
\item For all $v\in X$, $G_{(v)}=G_{(X)}$.
\item For all $g,h\in G$, $[g,h]\in G_{(X)}$
\end{enumerate}
\end{obs}

\subsection{The case of the binary tree}
We here consider the bounded automata groups $G$ that have a faithful representation in reduced form on $[2]^*$.  Any faithful representation $G=\grp{Y}\leq\Aut([2]^*)$ in reduced form is necessarily of abelian wreath type.

Let us first list all possible non-trivial automorphisms of $[2]^*$ that can appear as the distinguished generators of a bounded automata group $G=\grp{Y}\leq \Aut([2]^*)$ in reduced form. We give the automorphisms in wreath recursion.

\begin{obs}
Letting $\sigma\in \Sym(2)$ be the non-trivial element, any non-trivial generator $g$ of a bounded automata group $G=\grp{Y}\leq\Aut([2]^*)$ in reduced form  has one of the following types, up to taking an inverse:
\begin{enumerate}[(I)]
\item $g=\sigma$,
\item $g=(\sigma,g)$ or $g=(g,\sigma)$,
\item $g=\sigma(1,g)$, or
\item $g=\sigma(\sigma,g)$.
\end{enumerate}
\end{obs}

\begin{lem}\label{lem:type_4} If $G=\grp{Y}\leq\Aut([2]^*)$ is a bounded automata group in reduced form and $G$ contains an element of type $(IV)$, then $G$ is not elementary amenable.
\end{lem}

\begin{proof}
The group $G$ must contain $\sigma$, since it is self similar and contains an element of type $(IV)$. Say that $g=\sigma(\sigma,g)$ is an element of type $(IV)$ in $G$. It suffices to show $H:=\grp{\sigma,g}$ is not elementary amenable.

Let us suppose toward a contradiction that $H$ is elementary amenable. The group $H$ is a self-replicating and clearly non-abelian, so we apply Corollary~\ref{cor:ext_EA} to find a weakly self-replicating $M\normal G$ such that $\rk(M)+1=\rk(G)$. In view of Lemma~\ref{lem:fix}, we infer that $M$ acts transitively on $[2]$.

We see that $\sigma g=(\sigma,g)$, so $(\sigma g)^2=(1,g^{2})$. Taking $m\in M$ such that $m(0)=1$, we have that $[(1,g^{2}),m]\in M$, and as $M$ is weakly self-similar, $[(1,g^{2}),m]_1=g^{2}\in M$. The square $g^2$ equals $(g\sigma,\sigma g)$, so $\sigma g=(\sigma,g)\in M$, using again that $M$ is weakly self-similar. A final application of weak self-similarity implies that $\sigma$ and $g$ are elements of $M$. This is absurd, since $M$ must be a proper subgroup.
\end{proof}

\begin{lem}\label{lem:type_2+3}
If $G\leq\Aut([2]^*)$ is a bounded automata group of abelian wreath type and $G$ contains elements of type $(II)$ and $(III)$, then $G$ is not elementary amenable.
\end{lem}

\begin{proof}
The group $G$ must contain $\sigma$, since it is self similar and contains an element of type $(II)$. As the proofs are the same, let us assume that $G$ contains $g=(\sigma,g)$. Let us also assume that the type $(III)$ element is $h=\sigma(1,h)$. It now suffices to show $H:=\grp{\sigma,g,h}$ is not elementary amenable.

Let us suppose toward a contradiction that $H$ is elementary amenable. The group $H$ is a self-replicating and clearly non-abelian, so we apply Corollary~\ref{cor:ext_EA} to find a weakly self-replicating $M\normal G$ such that $\rk(M)+1=\rk(G)$. In view of Lemma~\ref{lem:fix}, we infer that $M$ acts transitively on $[2]$.

We  see that $\sigma h=(1,h)$. Taking $m\in M$ such that $m(0)=1$, we have that $[(1,h),m]\in M$, and as $M$ is weakly self-similar, $[(1,h),m]_1=h\in M$. The commutator $[g,h]$ is then also an element of $M$. Moreover,
\[
[g,h]=ghg^{-1}h^{-1}=(\sigma,g)\sigma(1,h)(\sigma,g^{-1})\sigma(h^{-1},1)=(\sigma h g^{-1}h^{-1},g\sigma)
\]
we conclude that $g\sigma \in M$, since $M$ is weakly self-replicating. As $h=\sigma(1,h)\in M$, it is then the case that $g\sigma h=(\sigma,g)(1,h)=(\sigma,gh)\in M$. From a second application of weak self-similarity, it follows that $\sigma$, $g$, and $h$ are elements of $M$. This is absurd as $M$ is a proper subgroup of $H$.
\end{proof}

\begin{rmk} The groups $H$ arising in the proofs of Lemma~\ref{lem:type_4} and \ref{lem:type_2+3} seem like they may be of independent interest. They appear to be the ``smallest" non-elementary amenable bounded automata groups.
\end{rmk}

\begin{thm}\label{thm:2-reg_BAG}
If $G=\grp{Y}\leq \Aut([2]^*)$ is a bounded automata group in reduced form, then either $G$ is virtually abelian or $G$ is not elementary amenable.
\end{thm}
\begin{proof}
Suppose that $G$ is elementary amenable. In view of Lemma~\ref{lem:type_4}, $Y$ contains no element of type $(IV)$. We now have two cases.

Suppose first that $Y$ contains one or two elements of type $(II)$. The case of one element is an easy adaptation of the proof for the case of two elements, so we only consider the latter. In view of Lemma~\ref{lem:type_2+3}, $Y$ does not contain an element of type $(III)$. It now follows that $Y$ has  three non-trivial elements: $Y=\{\sigma, (g,\sigma),(\sigma,h)\}$.  Seeing as $\sigma (\sigma ,g)\sigma=(h,\sigma)$, the group $G$ is generated by two involutions, $\sigma, (\sigma,g)$, so $G$ is either finite or the infinite dihedral group.

Suppose next that $Y$ contains an element of type $(III)$. As in the previous case, $Y$ does not contain an element of type $(II)$. The non-trivial elements of $Y$ then consist of either a single type $(III)$ element or a type $(III)$ element and a type $(I)$ element. In the former case, $G$ is  abelian, and in the latter case, $G$ is virtually abelian.
\end{proof}

\subsection{Elementary amenable groups}
We are now prepared to prove our main theorem of this section.

\begin{lem}\label{lem:F_non-trivial}
Suppose that $G=\grp{Y}\leq \Aut(\Xt)$ is a self-replicating group of abelian wreath type that is elementary amenable. If $G$ is non-abelian, then there is some directed $d\in Y$ and $x\in X$ such that $d_x$ is a non-trivial finitary element.
\end{lem}
\begin{proof}
We prove the contrapositive. If every directed  $d\in Y$ is such that $d_x\in\{1,d\}$ for all $x\in X$, then $G$ is a generalized basilica group, and in view of Proposition~\ref{prop:balanced}, $G$ is furthermore balanced. Theorem~\ref{thm:balanced_self-rep} thus implies that $G$ is abelian.
\end{proof}

Suppose that $G=\grp{Y}\leq\Aut(\Xt)$ is a self-replicating group of abelian wreath type with $|X|>2$ and suppose that $G$ is elementary amenable but not virtually abelian. Corollary~\ref{cor:ext_EA_3} supplies a weakly self-replicating $M\normal G$ such that $\rk(M)+1=\rk(G)$ and $G/M$ is finite.  Recall that $M$ acts transitively on $X$ by Lemma~\ref{lem:weakly_self-similar_trans}.  Setting $F:=\pi_1(G)\cap G$, Lemma~\ref{lem:factorization} below claims that $G=FM$; by Lemma~\ref{lem:F_non-trivial}, $F$ contains non-trivial elements. To prove this requires three preliminary lemmas, that amount to checking three cases.

Fix $c\in Y$ some directed element, let $v\in X$ be the active vertex of $c$ on $X$, and say that $O\subseteq X$ is the orbit of $v$ under the action of $\grp{c}$ on $X$.

\begin{lem}\label{lem:factorization1}
If there is $l\in F$ such that $lc$ fixes some $x\in X$, then there is $k\in F$ such that $kc\in M$.
\end{lem}
\begin{proof}
Set $d:=lc$. Since $G$ has abelian wreath structure, $d$ fixes $X$. If $d_w=1$ for some $w\in X$, then $c \in M$ by Lemma~\ref{lem:com_contain_sections_2}, since $d_{v}=c$. We thus suppose that $d_w\neq 1$ for all $w\in X$.  In view of Lemma~\ref{lem:F_non-trivial}, we may fix a non-trivial $f\in F$, and we may also fix $m\in M$ such that $m(v)\in X\setminus\{v,f(v)\}$. Such an $m$ exists since $|X|>2$ and $M$ is weakly self-replicating and of finite index. Set $y:=m_{v}$.

The commutator $[m^{-1},d^{-1}]$ is in $M$, so
\[
 [m^{-1},d^{-1}]_{v}=(m^{-1}d^{-1}md)_{v}=y^{-1}hyc
\]
is in $M$ for some $h\in F$. On the other hand, $[m^{-1},fd^{-1}f^{-1}]$ is also an element of $M$, so
\[
 [m^{-1},fd^{-1}f^{-1}]_{v}=y^{-1}h'yd_{f^{-1}(v)}
\]
is in $M$ for some $h'\in F$. The element $k:=d_{f^{-1}(v)}$ is also a non-trivial finitary element.

We now have that $c$ and a non-trivial finitary element $k$ are elements of $y^{-1}FyM$. The quotient $y^{-1}FyM/M$ is abelian, so $[k,c^{-1}]\in M$. As $k^{-1}$ moves $c(v)$, it follows that $[k,c^{-1}]_{v}=k'c$ for some $k'\in F$. Thus, $k'c\in M$, completing the proof.
\end{proof}

\begin{lem}\label{lem:factorization2}
If some non-trivial element of $F$ setwise stabilizes $O$, then there is $k\in F$ such that $kc\in M$.
\end{lem}
\begin{proof}
 Fix a non-trivial $f\in F$ such that $f$ stabilizes $O$. If $|O|\leq 2$, then $fc$ fixes pointwise $O$, and Lemma~\ref{lem:factorization1} implies that there is $k\in F$ such that $kc\in M$. We thus assume that $|O|>2$.

Let $j>1$ be least such that $c^j$ fixes $O$ and take $i<j$ least such that $fc^i$ fixes $O$ pointwise. If $i=j-1$, then $fc^{-1}$ fixes $O$ pointwise, and there is $k\in F$ such that $kc^{-1}\in M$ by Lemma~\ref{lem:factorization1}. In this case, $ck^{-1}\in M$, so conjugating by $k^{-1}$, we see that $k^{-1}c\in M$. We may thus assume that $i<j-1$. If $fc^i$ has a trivial section at some $z\in X$, then as in the previous lemma, Lemma~\ref{lem:com_contain_sections_2} implies that $kc\in M$ for some $k\in F$. We may thus also assume that all sections of $fc^i$ at $z\in X$ are non-trivial.

 The element $fc^{i}$ has non-finitary sections at
\[
\Omega:=\{v,c^{-1}(v),\dots,c^{-i+1}(v)\}.
\]
Note that $|O\setminus \Omega| \geq 2$, since $|\Omega|=i$ while $|O|=j$. We may find $m\in M$ such that
\[
m(v)\in O\setminus (\Omega\cup c^{-1}(\Omega)),
\]
since  $\Omega\cup c^{-1}(\Omega)$ is only one element larger than $\Omega$. Set $y:=m_{v}$.

The commutator $[m^{-1},(fc^{i})^{-1}]$ is in $M$, so
\[
 [m^{-1},(fc^{i})^{-1}]_{v}=y^{-1}h_1yh_0c
\]
is in $M$ for some $h_1,h_0\in F$. On the other hand, $c^{-1}fc^{i+1}$ has non-finitary sections exactly at $\Omega\cup c^{-1}(\Omega)$. In particular, $(c^{-1}fc^{i+1})_{v}=c^{-1}k_0c$ for some non-trivial $k_0\in F$, since we assume that $fc^i$ has no trivial sections. The commutator $[m^{-1},(c^{-1}fc^{i+1})^{-1}]$ is in $M$, so
\[
 [m^{-1},(c^{-1}fc^{i+1})^{-1}]_{v}=y^{-1}k_1yc^{-1}k_0c
\]
is in $M$, for some $k_1\in F$.

We now have that $h_0c$ and $c^{-1}k_0c$ are elements of $y^{-1}FyM$, hence $h_0c$ and $k_0$ are elements of $y^{-1}FyM$, since $F$ is an abelian group.  The quotient $y^{-1}FyM/M$ is abelian, so $[k_0,(h_0c)^{-1}]\in M$. As $k_0^{-1}$ moves $c(v)$, it follows that $[k_0,(h_0c)^{-1}]_{v}=k'c$ where $k'\in F$. Thus, $k'c\in M$, completing the proof.
\end{proof}

\begin{lem}\label{lem:factorization3}
If $c$ strongly active and no non-trivial element of $F$ setwise stabilizes $O$, then there is $k\in F$ such that $kc\in M$.
\end{lem}
\begin{proof} Suppose first that $\grp{c}$ has only two orbits $O_1=O$ and $O_2$ on $X$. The group $F$ therefore has exponent two. Take $i>1$ least such that $c^i$ fixes $O_1$. Let $m\in M$ be such that $m(v)\in O_2$ and set $y:=m_{v}$. We now see that $[m^{-1},c^{-i}]_{v}=y^{-1}h_1yhc$ is in $M$, where $h_1,h\in F$.  The group $M$ is normal, so $hcy^{-1}h_1y\in M$. Hence,
\[
hcy^{-1}h_1yy^{-1}h_1yhc=(hc)^2\in M.
\]
If $h=1$, then considering the section $[f,c^{-2}]_{v}$ for $f\in F$ non-trivial shows that $kc\in M$ for some $k\in F$, since $f^{-1}(O_1)=O_2$ and $c^2$ is not trivial. Let us thus suppose that $h$ is non-trivial, so $h(O_1)=O_2$. The element $(hc)^{2}$ has non-finitary sections exactly at $v$ and $c^{-1}h(v)$. The inverse $(hc)^{-2}$ therefore has non-finitary sections exactly at $c^2(v)$ and $ch(v)$.  We compute:
\[
[c^{-1},(hc)^{-2}]_{c^{-1}h(v)}=(c^{-1})_{h(v)}((hc)^{-2})_{c^2h(v)}c_{ch(v)}((hc)^{2})_{c^{-1}h(v)}.
\]
It now follows that $[c^{-1},(hc)^{-2}]_{c^{-1}h(v)}=l_1cl_0$ for $l_1,l_0\in F$. Conjugating by $l_0$, we see that there is $k\in F$ such that $kc\in M$.

Let us suppose finally that $\grp{c}$ has at least three orbits $O_1$, $O_2$, and $O_3$ and let $i$ be least such that $c^i$ fixes $O_1$. If $c^i$ has a trivial section, then Lemma~\ref{lem:com_contain_sections_2} ensures that $kc\in M$ for some $k\in F$. We may thus assume that every section of $c^i$ is non-trivial.  Fix $f\in F$ non-trivial, and without loss of generality, $f(O_1)=O_2$. Let $m\in M$ be such that $m(v)\in O_3$ and set $y:=m_{v}$. The commutator $[m^{-1},c^{-i}]$ is in $M$, so
\[
[m^{-1},c^{-i}]_{v}=y^{-1}h_1yh_0c
\]
is in $M$, where $h_1,h_0\in F$. The commutator $[m^{-1},fc^{-i}f^{-1}]$ is also in $M$, so
\[
[m^{-1},fc^{-i}f^{-1}]_{v}=y^{-1}k_1yk_0
\]
is in $M$ where $k_1,k_0\in F$, since $fc^if^{-1}$ has non-finitary sections only on $O_2$. Note also that $k_0$ is non-trivial since $c^i$ has every section non-trivial.

We now have that $h_0c$ and a non-trivial finitary element $k_0$ are elements of $y^{-1}FyM$. The quotient $y^{-1}FyM/M$ is abelian, so $[k_0,(h_0c)^{-1}]\in M$. As $k_0$ moves $v$, it follows that $[k_0,(h_0c)^{-1}]_{v}=k'c$ where $k'\in\grp{F}$. Thus, $k'c\in M$.
\end{proof}

Bringing together the previous three lemmas, we obtain the desired result.
\begin{lem}\label{lem:factorization}
Suppose that $G=\grp{Y}\leq \Aut(\Xt)$ with $|X|>2$ is a self-replicating group of abelian wreath type. If $G$ is elementary amenable but not virtually abelian, then the weakly self-replicating $M\normal G$ supplied by Corollary~\ref{cor:ext_EA_3} is such that $(\pi_1(G)\cap G)M=G$ and $\rk(M)+1=\rk(G)$. In particular, $G/M$ is finite and abelian.
\end{lem}
\begin{proof}
Let $M\normal G$ be the finite index subgroup given by Corollary~\ref{cor:ext_EA_3} and take $d\in Y$ directed. If $d$ is not strongly active, then $f:=1\in F:=\pi_1(G)\cap G$ is such that $fd$ fixes a vertex in $X$. Lemma~\ref{lem:factorization1} thus gives $k\in F$ such that $kd\in M$. If $d$ is strongly active, then either there is some non-trivial $f\in F$ such that $f$ setwise stabilizes the orbit of $v^d\in X$ under the action of $\grp{d}$ acting on $X$ or there is not. In the former case, Lemma~\ref{lem:factorization2} supplies $k\in F$ such that $kd\in M$ and in the latter, Lemma~\ref{lem:factorization3} gives $k\in F$ such that $kd\in M$.

It now follows that $FM=G$, and the lemma is verified.
\end{proof}

The desired theorem is now in hand.
\begin{thm}\label{thm:aws_groups}
If $G=\grp{Y}\leq \Aut(\Xt)$ is a self-replicating group of abelian wreath type, then either $G$ is virtually abelian, or $G$ is not elementary amenable.
\end{thm}
\begin{proof}
The case that $|X|=2$ is settled by Theorem~\ref{thm:2-reg_BAG}. We thus assume that $|X|>2$. Set $F:=\pi_1(G)\cap G$. We suppose toward a contradiction that $G$ is elementary amenable but not virtually abelian.

 Applying Corollary~\ref{cor:ext_EA_3}, we obtain non-trivial weakly self-replicating $L\normal G$ and $M\normal G$ such that $L\leq M$, $G/M$ is finite, $M/L$ is abelian, and $\rk(L)+2=\rk(M)+1=\rk(G)$. Lemma~\ref{lem:factorization} ensures further that $FM=G$. We now argue that $G/L$ is abelian, which contradicts the rank of $G$.

We have two cases: (1) $M\cap F\neq\{1\}$, and (2) $M\cap F=\{1\}$. For case (1), fix a non-trivial $f\in F\cap M$. By Lemma~\ref{lem:factorization}, each $d\in Y$ directed admits $h\in F$ such that $hd\in M$. Taking $hd\in M$, $[f,(hd)^{-1}]\in L$, so $[f,(hd)^{-1}]_{v^d}=kd$  is an element of $L$ where $k\in F$. We conclude that for each directed $d\in Y$ there is $h\in F$ such that $hd\in L$. It is thus the case that $FL=G$, so $G/L$ is abelian, giving the desired contradiction.

For case (2), fix $d\in Y$ directed and let $h\in F$ be such that $hd\in M$. If $hd$ fixes $v^d$, then $(hd)_{v^d}=d\in M$. Hence $h\in F\cap M=\{1\}$.  The element $d$ must be strongly active since else since $d\neq 1$, it has a finitary section that belongs to $M$ and so $M$ contains a non-trivial element of $F$. We conclude that $d$ does not fix $v^d$. We may thus assume that $hd$ does not fix $v^d$. Set $\tilde{d}:=hd$ and observe that $\tilde{d}$ has exactly one non-finitary section, namely $\tilde{d}_{v^d}=d$.

If $\grp{\tilde{d}}$ acts intransitively on $X$, let $i$ be least such that $\tilde{d}^i$ fixes $X$. Since $M$ intersects $F$ trivially, it must be the case that $\tilde{d}^i$ has trivial sections at all $w$ outside the orbit of $v^d$ under $\grp{\tilde{d}}$. The group $M$ acts transitively on $X$, so Lemma~\ref{lem:com_contain_sections_2} implies that $h'd\in L$ for some $h'\in F$.

Suppose that $\grp{\tilde{d}}$ acts transitively on $X$. Fix $f\in F$ that is non-trivial; such an element exists by Lemma~\ref{lem:F_non-trivial}. If $f(v^d)=\tilde{d}(v^d)$, then $f$ acts like a $|X|$-cycle on $X$, since $\pi_1(G)$ is abelian. Since $|X|>2$, $f^2$ is non-trivial, and $f^2(v^d)\neq \tilde{d}(v^d)$. By possibly replacing $f$ with $f^2$, we may assume that $f(v^d)\neq \tilde{d}(v^d)$.  The elements $f\tilde{d}^{-1}f^{-1}$ and $\tilde{d}$ are in $M$, so $[f\tilde{d}^{-1}f^{-1},\tilde{d}^{-1}]\in L$. Moreover,
\[
[f\tilde{d}^{-1}f^{-1},\tilde{d}^{-1}]_{v^d}=(\tilde{d}^{-1})_{f^{-1}\tilde{d}(v^d)}(\tilde{d}^{-1})_{\tilde{d}^2(v^d)}\tilde{d}_{f^{-1}\tilde{d}(v^d)}d = h'd,
\]
where $h'\in F$.

We now conclude that for all $d\in D$, there is $h\in F$ such that $hd\in L$. Hence, $FL=G$, and so $G/L$ is abelian, giving the desired contradiction.
\end{proof}

The infinite dihedral group shows that virtually abelian groups indeed arise as self-replicating groups of abelian wreath type; see Example~\ref{ssec:dihedral}.

\subsection{GGS groups}
We pause here to recall the GGS groups and observe that Theorem~\ref{thm:aws_groups} applies to these groups. See \cite{Ba93} for more information on these groups.

For $p$ an odd prime, form the $p$-regular rooted tree $[p]^*$. Let $a\in \Aut([p]^*)$ be such that $a(iz)=((i+1)\mod p) z$ for word $iz\in [p]^*$, where $i\in [p]$. The automorphism $a$ cyclically permutes the first level of the tree and acts rigidly below the first level. Take a $p-1$-tuple $\alpha:=(e_0,\dots,e_{p-2})$ where $e_i\in\{0,\dots, p-1\}$ for each $0\leq i\leq p-2$. Using wreath recursion, we define $b_{\alpha}\in \Aut([p]^*)$ by $b_{\alpha}:=(a^{e_1},\dots,a^{e_{p-2}},b)$.
%
%That is to say, for a word $iz\in [p]^*$ with $i\in [p]$,
%\[
%b_{\alpha}(iz)=
%\begin{cases}
%ia^{e_i}(z) & \text{ if } i<p-1\\
%ib(z) & \text{ if }i=p-1.
%\end{cases}
%\]

\begin{defn}
For $\alpha=(e_1,\dots,e_{p-1})$ a $p-1$-tuple where $e_i\in\{0,\dots, p-1\}$ for each $1\leq i\leq p-1$, the \textbf{GGS group} associated to $\alpha$ is the group $G_{\alpha}:=\grp{a,	b_{\alpha}}\leq \Aut([p]^*)$.
\end{defn}

\begin{prop}
If some coordinate of $\alpha$ is non-zero, then the $GGS$ group $G_{\alpha}$ is a self-replicating group of abelian wreath type.
\end{prop}
\begin{proof}
We see that $\pi_1(G_{\alpha})=\grp{a}=C_p$. Thus, $G_{\alpha}$ is of abelian wreath type. That $G_{\alpha}$ is self-replicating is also clear, and well-known.
\end{proof}

\begin{cor}
The only elementary amenable GGS groups are the cyclic groups $C_p$.
\end{cor}

\section{Examples}

\subsection{Balanced generalized basilica groups}\label{ssec:balanced}
Define $a,b\in \Aut([5]^*)$ by wreath recursion:
\[
a:=(01)(34)(1,1,1,1,a) \text{ and } b:=(12)(34)(1,1,1,b,1).
\]
The group $G:=\grp{a,b}\leq \Aut([5]^*]$ is a balanced generalized basilica group. On $[5]$, $G$ has two orbits $U:=\{0,1,2\}$ and $V:=\{3,4\}$, and $v^a,v^b$ are elements of $V$.

Setting $c:=ab$, an easy computation shows that $c=(012)(1,1,1,c,1)$. The element $c$ fixes $V$, and it follows that $c\in K:=\ker(G\acts V^*)$. In view of Lemma~\ref{lem:almost_self-replicating}, $K$ is locally finite, and since $G$ is two generated, it follows that $G/K\simeq \Zb$. The group $G$ is thus elementary amenable.

\begin{claim*}
 The set $\{b^{-n}cb^n\mid n\geq 0\}$ consists of pairwise distinct elements of $K$.
\end{claim*}
\begin{proof}
Noting that $b^2=(1,1,1,b,b)$, it is immediate that
\[
b^{-2n}cb^{2n}=(012)(1,1,1,b^{-n}cb^{n},1).
\]
An easy calculation shows that
\[
b^{-2n-1}cb^{2n+1}=(021)(1,1,1,1,b^{-n}cb^{n}).
\]
Arguing by induction on $i$ for the claim that $\{b^{-n}cb^n\mid i\geq n\geq 0\}$ consists of pairwise distinct elements proves the claim.
\end{proof}

The claim implies the kernel $K$ is infinite. Since all subgroups of a finitely generated virtually abelian group are finitely generated, we deduce that $G$ is \textit{not} virtually abelian. This example shows that Corollary~\ref{cor:balanced} is sharp.

We next exhibit a large family of generalized basilica groups.

\begin{defn}
For a prime $p\geq 3$, let $c$ be the cycle $(0\dots p-1)$ and define $g_{p},h_{p}\in \Aut([p]^*)$ by $g_{ p}:=c(g_{ p},1,\dots,1)$ and $h_{p}:=c(1,h_{p},1,\dots,1)$. We define $G_p:=\grp{g_{p},h_{p}}$.
\end{defn}
\begin{prop}\label{prop:G_p-properties}
For each prime $p\geq 3$, $G_p$ is a self-replicating balanced generalized basilica group, weakly branch, and not elementary amenable.
\end{prop}
\begin{proof}
It is immediate that $G:=G_p$ is a balanced generalized basilica group, since $g:=g_p$ and $h:=h_p$ act transitively on $[p]$. The elements $g^p$ and $h^p$ are such that $(g^p)_x=g$ and $(h^p)_x=h$ for all $x\in X$. It follows that $G_p$ is self-replicating. In view of Corollary~\ref{cor:trans_virt-ab}, we need only to check that $G$ is not abelian to conclude that $G$ is not elementary amenable. This is immediate by considering sections:
\[
(gh)_0=1\text{ and }(hg)_0=hg.
\]
It is clear that $hg$ is non-trivial, so $g$ and $h$ do not commmute.

Let us finally verify that $G$ is weakly branch. The commutator $[g^p,h]$ fixes $X$ and
\[
[g^p,h]_i=
\begin{cases}
[g,h] & i=2\\
1 & \text{else}
\end{cases}
\]
The rigid stabilizer $\rist_{G}(2)$ is non-trivial, and it follows that $\rist_G(i)$ is non-trivial for all $i\in X$. One verifies that $[g^{p^i},h]$ lies in $\rist_G(2^i)$ and $[g^{p^i},h]_{2^i}=[g,h]$, where $2^i$ is the word in $[p]^*$ consisting of $i$ many $2$'s. We infer that all rigid stabilizers are infinite. Since $G$ acts transitively on each level, $G$ is weakly branch.
\end{proof}

To distinguish the $G_p$ requires a result from the literature on weakly branch groups. For a tree $X^*$, the \textbf{boundary} of $X^*$, denoted by $\partial X^*$, is the collection of infinite words $X^{\omega}$ with the product topology. The group $\Aut(\Xt)$ has an action on $\partial \Xt$ by homeomorphisms which is induced from the action $\Aut(\Xt)\acts \Xt$. The product measure on $\partial \Xt$ arising from the uniform measure on $X$ is preserved by $\Aut(\Xt)$. This measure is called the \textbf{Bernoulli measure}.

\begin{thm}[{See \cite[Theorem 2.10.1]{N05}}]\label{thm:boundary_iso_wbranch}
Suppose  that $G\leq \Aut(X^*)$ and $H\leq \Aut(Y^*)$ are weakly branch groups. If $\phi: G\rightarrow H$ is an isomorphism of groups, then there is a measure preserving homeomorphism $F:\partial X^*\rightarrow \partial Y^*$ such that $\phi(g)(F(\alpha))=F(g(\alpha))$ for all $\alpha\in \partial \Xt$, where $\partial \Xt$ and $\partial Y^*$ are equipped with the respective Bernoulli measures
\end{thm}

\begin{thm}\label{thm:size_GBG}
For $p\neq q$ with $p,q\geq 3$, $G_p$ and $ G_q$ are not isomorphic.
\end{thm}
\begin{proof}
Suppose for contradiction there are $p\neq q$ such that $G_p$ is isomorphic to $G_q$. Appealing to Theorem~\ref{thm:boundary_iso_wbranch}, there is a measure preserving homeomorphism $F:\partial [p]^*\rightarrow \partial [q]^*$.

Fix $O\subseteq \partial [p]^*$ clopen with measure $\frac{1}{p}$. The image $F(O)$ is a clopen set in $\partial[q]^*$ with measure $\frac{1}{p}$. We may then find disjoint basic clopen sets $W_1,\dots,W_n$ of $\partial [q]^*$ such that $F(O)=\sqcup_{i=1}^nW_i$. The sets $W_i$ are basic, so for each $1\leq i \leq n$ there is $k_i\geq 1$ such that $\mu(W_i)=\frac{1}{q^{k_i}}$. We deduce that
\[
\frac{1}{p}=\sum_{i=1}^n\frac{1}{q^{k_i}}.
\]
Letting $m:=\max\{k_i\mid 1\leq i\leq n\}$, we see that
\[
\frac{q^m}{p}=\sum_{i=1}^{n}q^{m-k_i},
\]
and the right hand side is an integer. This implies that $p$ divides $q$ which is absurd.
\end{proof}

\subsection{The infinite dihedral group}\label{ssec:dihedral}
\begin{example}\label{ex:dihedral}
Let $X:=[3]$ and via wreath recursion, define $a,b\in \Aut(X^*)$  by $a:=(01)(1,1,a)$ and $b:=(02)(1,b,1)$. The group $G:=\grp{a,b}$ is a self-replicating generalized basilica group isomorphic to the infinite dihedral group.
\end{example}
In regards to Question~\ref{qu:main}, Example~\ref{ex:dihedral} shows that virtually abelian is the best one can hope for in the case of self-replicating generalized basilica groups.

\begin{example}\label{ex:dihedral_AWS}
Let $X:=[2]$ and via wreath recursion, define $a,b\in \Aut(X^*)$  by $a:=(01)(1,1)$ and $b:=(a,b)$. The group $G:=\grp{a,b}$ is a self-replicating bounded automata group of abelian wreath type and is isomorphic to the infinite dihedral group.
\end{example}
Example~\ref{ex:dihedral_AWS} shows that Theorem~\ref{thm:aws_groups} is sharp. One naturally wonders if the infinite dihedral group is the only possible non-abelian, but virtually abelian example.

\subsection{The lamplighter group}\label{ex:lamp}
Let $X:=[4]$ and form $\Aut(\Xt)$. Define elements of $\Aut(\Xt)$ by $a:=(02)(13)(1,1,1,a)$ and $b:=(02)(1,b,1,1)$. One easily verifies that $G:=\grp{a,b}\leq \Aut(\Xt)$ is a bounded automata group. In fact $G$ is a bounded automata group that is both a generalized basilica group and of abelian wreath type.

The element $b$ has order $2$, and a simple calculation shows the following:
\[
a^{2n}ba^{-2n}=(02)(1,a^nba^{-n},1,1) \text{ and } a^{2n+1}ba^{-2n-1}=(02)(1,1,1,a^nba^{-n})
\]
for all $n\geq \Zb$.

We now argue that $C_2\wr \Zb$ is isomorphic to $G$; recall that $C_2$ denotes $\Zb/2\Zb$. This requires two claims.
\begin{claim*}
For any $n\neq m$ in $\Zb$, $a^nba^{-n}$ and $a^mba^{-m}$ commute
\end{claim*}
\begin{proof}
It suffices to show that $a^nba^{-n}$ and $a^mba^{-m}$ commute for all $n,m\geq 0$. Set $\Omega_n:=\{a^iba^{-i}\mid 0\leq i \leq n\}$. We argue by induction on $n$ that $a^{n+1}ba^{-n-1}$ commutes with every $g\in \Omega_n$. For the base case, $n=0$, we see
\[
\begin{array}{rcl}
baba^{-1} & = &(02)(1,b,1,1)(02)(1,1,1,b)\\
			& =& (1,b,1,b)\\
			& =& (02)(1,1,1,b)(02)(1,b,1,1)\\
			& =& aba^{-1}b,
\end{array}
\]
and thus $b$ and $aba^{-1}$ commute.

Suppose the inductive claim holds for $n$ and take $a^mba^{-m}\in \Omega_{n+1}$. The cases that $m$ is odd while $n+1$ is even and $m$ is even while $n+1$ odd follow as the base case.  Let us now suppose that $m$ and $n+1$ are both even; the case where they are both odd is similar. We then have that $a^{n+2}ba^{-n-2}=(1,a^kba^{-k},1,1)$ where $k<n+2$ and that $a^mba^{-m}=(02)(1,a^lba^{-l},1,1)$ where $l<n+2$. We see that
\[
a^{n+2}ba^{-n-2}a^mba^{-m}=(1,a^kba^{-k}a^lba^{-l},1,1)
\]
and
\[
 a^mba^{-m}a^{n+2}ba^{-n-2}=(1,a^lba^{-l}a^kba^{-k},1,1).
\]
The inductive hypothesis ensures that $a^kba^{-k}a^lba^{-l}=a^lba^{-l}a^kba^{-k}$. Therefore, $a^mba^{-m}$ and $a^{n+2}ba^{-n-2}$ commute, completing the induction.
\end{proof}

\begin{claim*} For any distinct $i_1,\dots,i_n$ in $\Zb$, the product $\prod_{j=1}^na^{i_j}ba^{-i_j}$ is not $1$.
\end{claim*}
\begin{proof}
Suppose for a contradiction that the claim is false, and let $n$ be least such that there are distinct $i_1,\dots,i_n$ for which $g:=\prod_{j=1}^na^{i_j}ba^{-i_j}=1$. By conjugating with an appropriate power of $a$, we may assume that all $i_j$ are positive. We may also take the $i_1,\dots,i_n$ to be such that $\sum_{j=1}^ni_j$ is least.

The section of $g$ at $1$ must have the form $\prod_{j=1}^sa^{k_j}ba^{-k_j}$. Likewise, the section of $g$ at $3$ must have the form $\prod_{j=1}^ra^{l_j}ba^{-l_j}$. These two sections must also be trivial, and since $n\geq 2$, it cannot be the case that $r=0=s$. We may thus assume without loss of generality that $s\geq 1$. Since we choose $n$ to be least, it is indeed the case that $s=n$. However, $l_j<i_j$, so $\sum_{j=1}^nl_j<\sum_{j=1}^ni_j$. This contradicts out choice of $i_1,\dots,i_j$. This completes the reductio.
\end{proof}

\begin{prop}
The lamplighter group $C_2\wr \Zb$ is isomorphic to $G$.
\end{prop}

\begin{proof}
We have that $C_2\wr\Zb=\bigoplus_{\Zb}C_2\rtimes \Zb$. We define $\Phi:\bigoplus_{\Zb}C_2\rtimes \Zb\rightarrow G$ by
\[
(f,n)\mapsto \left( \prod_{i\in \Supp(f)}a^iba^{-i}\right)a^n
\]
where $\Supp(f):=\{i\in \Zb\mid f(i)\neq 0\}$.

Let us verify $\Phi$ is a homomorphism:
\[
\begin{array}{rcl}
\Phi((f,n)(g,m)) & = & \Phi((f+n.g,n+m))\\
					& = & \left( \prod_{i\in \Supp(f+n.g)}a^iba^{-i} \right)a^{n+m}\\
					& = & \left(\prod_{i\in \Supp(f)}a^iba^{-i}\prod_{i\in \Supp(n.g)}a^iba^{-i}\right) a^{n+m}\\
					& = & \left(\prod_{i\in \Supp(f)}a^{i}ba^{-i}\prod_{i\in \Supp(g)}a^{i+n}ba^{-i-n}\right)a^{n+m}\\
					& = & \left(\prod_{i\in \Supp(f)}a^{i}ba^{-i}\right) a^n \left(\prod_{i\in \Supp(g)}a^{i}ba^{-i}\right)a^{m}\\
					& = & \Phi((f,n))\Phi((g,m)).			
\end{array}
\]
The third equality uses that the conjugates of $b$ by powers of $a$ commute and that $b$ has order two.

It is clear that $\Phi$ is surjective. For injectivity, if $\left( \prod_{i\in \Supp(f)}a^iba^{-i}\right)a^n=1$, then $n=0$, and $\prod_{i\in \Supp(f)}a^iba^{-i}=1$. The second claim ensures that $ \prod_{i\in \Supp(f)}a^iba^{-i}=1$ exactly when the product is trivial. It follows that $\Phi$ is injective.
\end{proof}

\subsection{Weak self-replication}\label{ssec:weak_self-rep}

Let $X:=[2]$ and form $\Aut(\Xt)$. Let $\sigma$ be the cycle $(01)$ and define $a:=\sigma (\sigma,\sigma a)$.  The group $G:=\grp{a}$ is isomorphic to $\Zb$, so it cannot contain the section $\sigma$. On the other hand, one checks that $\sigma a \sigma =a^{-1}$, so $a^2=(a,a^{-1})$. The section homomorphism $\phi_x:G_{(x)}\rightarrow \Aut(\Xt)$ thus has image $G$. The group $G$ is weakly self-replicating but not self-similar, so it is not self-replicating.

%=========================== The bibliography=================================================

\bibliographystyle{amsplain}
\bibliography{biblio}

\end{document}